\documentclass[11pt,twoside]{article}
\usepackage{amsmath,amssymb,epsfig}
\usepackage{amsmath,amssymb,epsfig,color,graphicx,subfigure}
\usepackage{multirow}

\newtheorem{proposition}{Proposition}[section]
\newtheorem{theorem}[proposition]{Theorem}
\newtheorem{lemma}[proposition]{Lemma}

\newtheorem{definition}[proposition]{Definition}
\newtheorem{remark}[proposition]{Remark}
\newtheorem{example}[proposition]{Example}
\newtheorem{algorithm}[proposition]{Algorithm}

\newenvironment{proof}{{\noindent \bf Proof:}}{\hfill $\fbox{}$ \vspace*{5mm}}
\numberwithin{equation}{section}

\newcommand{\bfe}{{\bf e}}

\newcommand{\bp}{{\bf p}}

\newcommand{\bu}{{\bf u}}

\newcommand{\br}{{\bf r}}

\newcommand{\bv}{{\bf v}}

\newcommand{\cg}{\mathcal{G}}

\newcommand{\ch}{\mathcal{H}}

\newcommand{\grass}{{\rm Grass}}

\newcommand{\st}{{\rm St}}
\newcommand{\tr}{{\rm tr}}

\newcommand{\diag}{{\rm diag}}

\newcommand{\ve}{{\rm vec}}

\newcommand{\C}{{\mathbb C}}

\newcommand{\Cn}{{\mathbb C}^{n}}

\newcommand{\Rnn}{{\mathbb R}^{n\times n}}

\newcommand{\BE}{\begin{equation}}
\newcommand{\EE}{\end{equation}}

\DeclareMathOperator*{\argmin}{argmin}
\newcommand{\normmm}[1]{{\vert\kern-0.25ex \vert\kern-0.25ex \vert #1
    \vert\kern-0.25ex \vert\kern-0.25ex\vert}}

\begin{document}

\title{New Formulation and Computation for Generalized Singular Values of Grassman Matrix Pair}
\author{Wei-Wei Xu\thanks{School of Mathematics and Statistics, Nanjing University of Information Science and Technology, Nanjing 210044, People's Republic of China (wwx19840904@sina.com). The research of this author is partially supported by the National Natural Science Foundation of China  (No. 11971243) and the Natural Science Foundation of Jiangsu Province of China (No. BK20181405).}
\and Michael K. Ng\thanks{Department of Mathematics, The University of Hong Kong, Pokfulam, Hong Kong, People's Republic of China (mng@maths.hku.hk). The research of this author is partially supported by the HKRGC GRF 12306616, 12200317, 12300218, 12300519, and HKU 104005583.}
\and Zheng-Jian Bai\thanks{School of Mathematical Sciences and Fujian Provincial Key Laboratory on Mathematical Modeling \& High Performance Scientific Computing,  Xiamen University, Xiamen 361005, People's Republic of China (zjbai@xmu.edu.cn). The research of this author is partially supported by the National Natural Science Foundation of China (No. 11671337) and the Fundamental Research Funds for the Central Universities (No. 20720180008).}
}

\maketitle

\begin{abstract}
In this paper, we derive new model formulations for computing generalized singular values of a Grassman matrix pair.
These new formulations make use of truncated filter matrices to locate the $i$-th generalized singular value of a Grassman matrix pair.
The resulting matrix optimization problems can be solved by using numerical methods involving Newton's method on Grassmann manifold. 
Numerical examples on synthetic data sets and gene expression data sets are reported to demonstrate the high accuracy
and the fast computation of the proposed new ormulations  for computing arbitrary generalized singular value of Grassman matrix pair.
 \end{abstract}


\vspace{3mm}

\pagestyle{myheadings} \thispagestyle{plain} \markboth{}{New model formula and methods for
computing generalized singular values of GMP}

\section{Introduction}

The generalized singular value decomposition (GSVD) used in mathematics and numerical computations \cite{1,2,3} is a very useful and versatile tool. The GSVD of two matrices having the same number of columns was first proposed by Van Loan \cite{3}.
It is  very useful in many matrix computation problems and practical applications, such as the Kronecker canonical form of a general matrix pencil, the linearly constrained least-squares problem, the general Gauss-Markov linear model, the generalized total least-squares problem, real-time signal processing, comparative analysis of DNA microarrays and so on; see for example \cite{4,5,6,7, 8,9,13,1, 2,10, 11,15,3,16,17,14}.

Numerical methods and perturbation analysis of GSVD have been well developed; see for instance \cite{6,9,13, 10,11,15,23,3,16,17,14}.
The GSVD of two matrices having the same number of columns was first proposed by Van Loan \cite{3}. Van Loan \cite{3} and Paige \cite{11} provided algorithms for computing the generalized singular value decomposition. Bai and Demmel \cite{6} described a variation of Paige's algorithm for computing the GSVD introduced by Van Loan \cite{3} and Paige and Saunders \cite{PS81}. Ewerbring and Luk \cite{13} and Zha \cite{14} proposed a GSVD for matrix triplets. Stewart \cite{15} and Van Loan \cite{16} proposed two algorithms for computing the GSVD. On perturbation analysis of the GSVD, Sun \cite{23} and Li \cite{10} presented several perturbation bounds of generalized singular values (GSVs) of a Grassman matrix pair (GMP) and their associated subspaces. Xu et al. \cite{17} provided the explicit expression and sharper bounds of the chordal metric between GSVs of a GMP.

Recently,  the GSVD plays an important role in the analysis of DNA microarrays and gene data. For arbitrary generalized singular value $(\alpha_{i},\beta_{i})$ of a GMP, the ratio $\alpha_{i}/\beta_{i}$ can be used in comparison analysis of gene data as an indicator; see for example \cite{4,5}. Moreover, the formulation for arbitrary GSV of a GMP has not been studied.
Usually, the GSVD are computed by making decompositions (e.g., QR decomposition and CS decomposition).
In these cases, computing the whole decomposition makes computational cost higher.
The main goal of this paper is to propose new model formulations for computing arbitrary GSV of a GMP. We first derive new formulations for computing arbitrary GSV of the GMP. By using truncated filter matrices, the matrix optimization problems can be reformulated to locate the $i$-th GSV of the GMP. The resulting optimization problems can be solved by using Newton's method on Grassmann manifold. Numerical examples on synthetic data sets and gene expression data sets are reported to demonstrate the high accuracy and the fast computation of the proposed method for computing
arbitrary  GSV of a GMP.

\subsection{Organization} The rest of this paper is organized as follows. In Section 2,
we provide mathematical preliminaries and derive new formulations for
computing GSVs of a GMP.
In Section 3, we present the numerical method for solving the resulting optimization models.
In Section 4, we provide numerical examples to show the efficiency of
the theoretical results. Finally, some concluding remarks are given in Section 5.

\subsection{Notation}
Throughout this paper we always use the following notation and definitions. $\imath$ denotes imaginary unit $\sqrt{-1}$ and $\mathbb{R}$, $\mathbb{C}$, $\mathbb{R}^{n}$, $\mathbb{C}^{m\times n}$ and $\mathbb{U}_n$ are the sets of real numbers, complex numbers, $n$-dimensional real vectors, $m\times n$ complex matrices and $n\times n$ unitary matrices accordingly. $|\cdot|$ stands for the absolute value of a complex number. $I_{n}$ and $O_{m\times n}$ stand for the identity matrix of order $n$ and the $m\times n$ zero matrix, respectively. $\bar{A}$, $A^{T}$, $A^{H}$, $A^{-1},$ $\det(A),\;\tr(A)$ denote the conjugate, transpose, conjugate transpose, inverse, determinant and trace of a matrix $A$ accordingly. By $\|\cdot \|_{2}$ we denote the spectral norm of a matrix. The singular value set of $A$ is denoted by $\sigma(A)$. For given matrices $A,B\in\mathbb{C}^{n\times n}$, $A<(\leq)B$ means $B-A$ is a positive (semi-)definite matrix.  For a matrix $A\in\mathbb{C}^{n\times n}$, we denote by $\sigma_1(A)\ge\sigma_2(A)\ge\cdots \sigma_n(A)\ge 0 $ its singular values, arranged in decreasing order. We denote $[\gamma]$ the largest integer less than or equal to a real number $\gamma$. The symbol ``$\otimes$''  means the Kronecher product and $\ve(\cdot)$ creates a column vector from a matrix by stacking its column vectors below one another.  Finally, $\diag(A,B)$ denotes a block diagonal matrix in which the diagonal blocks are square matrices $A$ and $B$.
\begin{definition}\protect{\cite{1}}
Let $A\in\mathbb{C}^{m\times n}$ and $B\in\mathbb{C}^{p\times n}$. A matrix pair $\{A,B\}$ is an $(m,p,n)$-GMP if rank $(A^{T},B^{T})^{T}=n$.
\end{definition}

For any $(m, p,n)$-GMP $\{A, B\}$, one may see $Z=(A^{T},B^{T})^{T}$ as a point of the complex projective space $G_{n,m+p}$ of all $n$-dimensional subspaces of the $(m+p)$-dimensional complex space $\mathbb{C}^{m+p}$, i.e., one can identify $Z$ with the linear subspace $R(Z)$ (see e.g. \cite{27,26,23}). Clearly, if $\{A, B\}$ is an $(m, p, n)$-GMP, then $(A^{H}A, B^{H}B)$ is a definite matrix pair, i.e., $x^{H}A^{H}Ax +x^{H}B ^{H}Bx>0$ for all nonzero $x\in\mathbb{C}^{n}$. The definite pair $(A^{H}A, B^{H}B)$ has $n$ generalized eigenvalues, and thus the GMP $\{A,B\}$ has $n$ generalized singular values.  A well-developed perturbation theory for the generalized eigenvalue problem of definite pencils is available. Perturbation bounds for the generalized singular value problem can be obtained with the help of the close relation between the two problems. However, the bounds obtained in this way are often unsatisfactory, just like the perturbation bounds for the singular values (SVs) of a single matrix A obtained through the perturbation bounds for the eigenvalues of the Hermitian matrix $A^{H}A$. Therefore, special attention deserves to be paid to perturbations for the generalized singular value problem. Therefore, research on numerical methods and perturbation analysis of the GSVD of a GMP is an important topic (see e.g. \cite{6,10, 23, 3,17}).

\begin{definition}\protect{\cite{1}}
Let $\{A,B\}$ be an $(m,p,n)$-GMP. A nonnegative number pair $(\alpha,\beta)$ is a GSV of
the GMP $\{A,B\}$ if
\[
(\alpha,\beta)\neq (0,0),\quad \det(\beta^2A^{H}A-\alpha^2B^{H}B)=0,\quad \alpha,\beta\ge 0.
\]
The set of GSV of $\{A,B\}$ is denoted by $\sigma\{A,B\}.$ Evidently,
\[
\sigma\{A,B\}=\{(\alpha,\beta)\neq(0,0)\;|\; \mathrm{det}(\beta^{2}A^{H}A-\alpha^{2}B^{H}B)=0,\;\alpha,\beta\geq0\}.
\]
\end{definition}

In the literature \cite{3,16}, there are several formulations of the GSVD. Here we adopt the following form.

\begin{definition}\label{def:gsvd}
Let \{A,B\} be an (m,p,n)-GMP. Then there exist unitary matrices
$U\in\mathbb{C}^{m\times m}, V \in\mathbb{C}^{p\times p},$ and a nonsingular matrix $R\in\mathbb{C}^{n\times n}$ such that
\begin{equation}\label{gsvd1}
U^{H}AR^{-1}=\Sigma_{A},\
V^{H}BR^{-1}=\Sigma_{B},
\end{equation}
\begin{equation}\label{gsvd2}
\Sigma_{A}=\begin{pmatrix}\Lambda&\\&O_{(m-r-s)\times(n-r-s)}\end{pmatrix},\
\Sigma_{B}=\begin{pmatrix}O_{(p+r-n)\times r} & \\ &\Omega\end{pmatrix},
\end{equation}
where 
\[
\Lambda=\diag(\alpha_{1},\ldots,\alpha_{r+s}),\;\Omega=\diag(\beta_{r+1},\ldots,\beta_{n}),
\]
with
\[
1=\alpha_{1}=\cdots=\alpha_{r}>\alpha_{r+1}\geq\cdots\geq\alpha_{r+s}>\alpha_{r+s+1}=\cdots=\alpha_{n}=0,
\]
\[
0=\beta_{1}=\cdots=\beta_{r}<\beta_{r+1}\leq\cdots\leq\beta_{r+s}<\beta_{r+s+1}=\cdots=\beta_{n}=1,
\]
and
\[
\alpha_{i}^{2}+\beta_{i}^{2}=1,\quad \; 1\leq i\leq n.
\]
\end{definition}

\section{Trace function optimization model formulation} In this section we give new model formulation of any GSV of a GMP  by trace function under one variable.
\begin{lemma}\label{lem:2.1} \protect{\cite{18}}
Let $c\in\mathbb{R}$ and $A_1, \ldots, A_m \in \mathbb{C}^{n\times n}$ with $\sigma_{1}(A_{i})\geq\sigma_{2}(A_{i})\geq\cdots\geq\sigma_{n}(A_{i}) $ for $i=1,\ldots,m$. Then we have 
\begin{eqnarray*}
&&\max\limits_{U_1,\ldots, U_m\in \mathbb{U}_{n}}\mathfrak{Re}[\tr(cI_{n}\pm\prod_{j=1}^{m}  U_j A_j)] =nc+\sum\limits_{i=1}^{n}\prod\limits_{j=1}^{m}\sigma_{i}(A_{j}),\\
&& \max\limits_{U_1,\ldots, U_m\in \mathbb{U}_{n}} |\tr(cI_{n}\pm\prod_{j=1}^{m}  U_j A_j)| =n\mid c\mid+\sum\limits_{i=1}^{n}\prod\limits_{j=1}^{m}\sigma_{i}(A_{j}),\\
&&\min\limits_{U_1,\ldots, U_m\in \mathbb{U}_{n}}\mathfrak{Re}[\tr(cI_{n}\pm\prod_{j=1}^{m}  U_j A_j)] = nc-\sum\limits_{i=1}^{n}\prod\limits_{j=1}^{m}\sigma_{i}(A_{j}),
\end{eqnarray*}
\begin{eqnarray*}
&&\min\limits_{U_1,\cdots, U_m\in \mathbb{U}_{n}} |\tr(cI_{n}\pm\prod_{j=1}^{m}  U_j A_j)| 
= \left\{ \begin{array}{ll} n\mid c\mid-\sum\limits_{i=1}^{n}\prod\limits_{j=1}^{m}\sigma_{i}(A_{j}), & \frac{1}{n}\sum\limits_{i=1}^{n}\prod\limits_{j=1}^{m}\sigma_{i}(A_{j})\leq \mid c\mid, \\
0, & \frac{1}{n}\sum\limits_{i=1}^{n}\prod\limits_{j=1}^{m}\sigma_{i}(A_{j})\geq\mid c\mid.\nonumber \end{array} \right.
\end{eqnarray*}
\end{lemma}

\begin{lemma}\label{lem:2.2} \protect{\cite{19}}
Let $f(X):\mathbb{C}^{n\times n}\rightarrow \mathbb{R}$ be an analytical function of several complex variables on the domain $XX^{H}\leq I_{n}$. Then $f(X)$ attains its maximum modulus on the characteristic manifold $\{X\in\mathbb{C}^{n\times n}:XX^{H}= I_{n}\}$.
\end{lemma}

We now give new formula model of any GSV of a GMP by trace function optimization under one variable.
\begin{lemma}\label{lem:2.3}
Let $\{A,B\}$ be an (m,p,n)-GMP and the GSV $(\alpha_{i},\beta_{i})$ of $\{A,B\}$
be given in \eqref{gsvd1} and \eqref{gsvd2}. The following conclusions hold true. 
\begin{itemize}
\item[\rm a)]
If $n\leq m,$ then for $1\leq i\leq n$,
\begin{eqnarray}\label{def:ai}
\alpha_{i}^{2} &=&\max\limits_{\Phi_{1}\in\mathbb{U}_{m}}\tr(A^{H}\Phi_{1}^{H}Q_{i}\Phi_{1} A(A^{H}A+B^{H}B)^{-1}) \nonumber\\
&&-\max\limits_{\Phi_{1}\in\mathbb{U}_{m}}\tr(A^{H}\Phi_{1}^{H}Q_{i-1}\Phi_{1} A(A^{H}A+B^{H}B)^{-1})\equiv \phi_{i},
\end{eqnarray}
where 
\BE\label{def:Qi}
Q_{i}=\diag(T_i, O_{(m-n)\times(m-n)})
\quad\mbox{with}\quad
T_{i}=\diag(I_i, O_{(n-i)\times(n-i)}).
\EE
\item[\rm b)]
If $n\leq p$, then for $1\leq i\leq n$,
\begin{eqnarray}\label{def:bi}
\beta_{i}^{2} &=&\max\limits_{\Phi_{2}\in\mathbb{U}_{p}}\tr(B^{H}\Phi_{2}^{H}P_{i}\Phi_{2} B(A^{H}A+B^{H}B)^{-1}) \nonumber\\
&&-\max\limits_{\Phi_{2}\in\mathbb{U}_{p}}\tr(B^{H}\Phi_{2}^{H}P_{i-1}\Phi_{1}B(A^{H}A+B^{H}B)^{-1})\equiv \psi_{i},
\end{eqnarray}
where
\BE\label{def:Pi}
P_{i}=\diag(O_{(p-n)\times(p-n)}, F_{i})
\quad\mbox{with}\quad
F_{i}=\diag(O_{(i-1)\times(i-1)}, I_{n-i+1}).
\EE

\item[\rm c)]
If $m\leq n$, then $\alpha_{m+1}=\cdots=\alpha_{n}=0$ and for $1\leq i\leq m$,
\begin{eqnarray}\label{def:ai-2}
\alpha_{i}^{2} &=& \max\limits_{\Psi_{1}\in\mathbb{U}_{m}}\tr(A^{H}\Psi_{1}^{H}\mathcal{Q}_{i}\Psi_{1} A(A^{H}A+B^{H}B)^{-1})\nonumber\\
&&-\max\limits_{\Psi_{1}\in\mathbb{U}_{m}}\tr(A^{H}\Psi_{1}^{H}\mathcal{Q}_{i-1}\Psi_{1} A(A^{H}A+B^{H}B)^{-1})\equiv \varphi_{i},
\end{eqnarray}
where
\BE\label{def:cqi}
\mathcal{Q}_{i}=\diag(I_i, O_{(m-i)\times(m-i)}).
\EE

\item[\rm d)]
If $p\leq n$, then $\beta_{1}=\cdots=\beta_{n-p}=0$ and for $n-p+1\leq i\leq n$,
\begin{eqnarray}\label{def:bi-2}
\beta_{i}^{2} &=& \max\limits_{\Psi_{2}\in\mathbb{U}_{p}}\tr(B^{H}\Psi_{2}^{H}\mathcal{P}_{i}\Psi_{2}B(A^{H}A+B^{H}B)^{-1})\nonumber\\
&&-\max\limits_{\Psi_{2}\in\mathbb{U}_{p}}\tr(B^{H}\Psi_{2}^{H}\mathcal{P}_{i-1}\Psi_{2}B(A^{H}A+B^{H}B)^{-1})\equiv \chi_{i}.
\end{eqnarray}
where
\BE\label{def:cpi}
\mathcal{P}_{i}=\diag(O_{(p-n+i-1)\times(p-n+i-1)}, I_{n-i+1}).
\EE
\end{itemize}
\end{lemma}
\begin{proof} a) If $n\leq m,$ then for $1\leq i\leq n$, we let $Q_{i}$ be defined by \eqref{def:Qi}. We set $\Sigma_{A}=
\begin{pmatrix}\hat{\Sigma}_{A}\\O_{(m-n)\times n}\end{pmatrix}$ and for any $\Phi_{1}\in\mathbb{U}_{m}$, we let
$\Phi_{1}U=\begin{pmatrix}\phi_{11} & \phi_{12}\\\phi_{21}&\phi_{22}\end{pmatrix}$  with $\phi_{11}\in\mathbb{C}^{n\times n}$. 
Using the GSVD of $\{A,B\}$ in \eqref{gsvd1} and \eqref{gsvd2} we have
\begin{align*}
&\max\limits_{\Phi_{1}\in\mathbb{U}_{m}}\tr(A^{H}\Phi_{1}^{H}Q_{i}\Phi_{1} A(A^{H}A+B^{H}B)^{-1})\\
=&\max\limits_{\Phi_{1}\in\mathbb{U}_{m}}\tr((A^{H}A+B^{H}B)^{-1/2}A^{H}\Phi_{1}^{H}Q_{i}\Phi_{1} A(A^{H}A+B^{H}B)^{-1/2})\\
=&\max\limits_{\Phi_{1}\in\mathbb{U}_{m}}\tr(\Sigma_{A}^{H}U^{H}\Phi_{1}^{H}Q_{i}\Phi_{1} U\Sigma_{A})\\
=&\max\limits_{\phi_{11}\phi_{11}^{H}\leq I_{n}}\tr(\hat{\Sigma}_{A}^{H}\phi_{11}^{H}T_{i}\phi_{11}\hat{\Sigma}_{A}).
\end{align*}
Observe $f(\phi_{11})=\tr(\hat{\Sigma}_{A}^{H}\phi_{11}^{H}T_{i}\phi_{11}\hat{\Sigma}_{A})$ is an analytical function of several complex variables on the domain $\phi_{11}\phi_{11}^{H}\leq I_{n}$. By Lemma \ref{lem:2.2} we have $f(\phi_{11})$
attains its maximum modulus on the characteristic manifold $\{\phi_{11}\in\mathbb{C}^{n\times n}:\phi_{11}\phi_{11}^{H}= I_{n}\}$. Therefore,
\[
\max\limits_{\phi_{11}\phi_{11}^{H}\leq I_{n}}\tr(\hat{\Sigma}_{A}^{H}\phi_{11}^{H}T_{i}\phi_{11}\hat{\Sigma}_{A})=\max\limits_{\phi_{11}\phi_{11}^{H}= I_{n}}\tr(\hat{\Sigma}_{A}^{H}\phi_{11}^{H}T_{i}\phi_{11}\hat{\Sigma}_{A}).
\]
Thus, by Lemma \ref{lem:2.1} we have
\BE\label{nf:ai}
\max\limits_{\Phi_{1}\in\mathbb{U}_{m}}\tr(A^{H}\Phi_{1}^{H}Q_{i}\Phi_{1}A(A^{H}A+B^{H}B)^{-1})=\max\limits_{\phi_{11}\phi_{11}^{H}= I_{n}}\tr(\hat{\Sigma}_{A}^{H}\phi_{11}^{H}T_{i}\phi_{11}\hat{\Sigma}_{A})
=\alpha_{1}^{2}+\cdots+\alpha_{i}^{2}.
\EE
Similarly, we have
\BE\label{nf:ai1}
\max\limits_{\Phi_{1}\in\mathbb{U}_{m}}\tr(A^{H}\Phi_{1}^{H}Q_{i-1}\Phi_{1} A(A^{H}A+B^{H}B)^{-1})=\alpha_{1}^{2}+\cdots+\alpha_{i-1}^{2}.
\EE
From \eqref{nf:ai}  and \eqref{nf:ai1} it follows that \eqref{def:ai} holds for $1\leq i\leq n$.

b) If $n\leq p$, then for $1\le i\le n$, let $P_{i}$ be defined by \eqref{def:Pi}. Let $\Sigma_{B}=
(O_{(p-n)\times n},\hat{\Sigma}_{B})$. For any $\Phi_{2}\in\mathbb{U}_{p}$, let $\Phi_{2}V=\begin{pmatrix}\psi_{11} & \psi_{12}\\ \psi_{21}&\psi_{22}\end{pmatrix}\in\mathbb{U}_{p}$ with $\psi_{22}\in\mathbb{C}^{n\times n}$. Using the GSVD of $\{A,B\}$ in \eqref{gsvd1} and \eqref{gsvd2} we have
\begin{align*}
&\max\limits_{\Phi_{2}\in\mathbb{U}_{p}}\tr(B^{H}\Phi_{2}^{H}P_{i}\Phi_{2} B(A^{H}A+B^{H}B)^{-1}) \\
&\max\limits_{\Phi_{2}\in\mathbb{U}_{p}}\tr((A^{H}A+B^{H}B)^{-1/2}B^{H}\Phi_{2}^{H}P_{i}\Phi_{2} B(A^{H}A+B^{H}B)^{-1/2})\\
=&\max\limits_{\Phi_{2}\in\mathbb{U}_{p}}\tr(\Sigma_{B}^{H}V^{H}\Phi_{2}^{H}P_{i}\Phi_{2}V\Sigma_{B})\\
=&\max\limits_{\psi_{22}\psi_{22}^{H}\leq I_{n}}\tr(\hat{\Sigma}_{B}^{H}\psi_{22}^{H}F_{i}\psi_{22}\hat{\Sigma}_{B}),
\end{align*}
which attains its maximum modulus on the characteristic manifold $\{\psi_{22}\in\mathbb{C}^{n\times n}:\psi_{22}\psi_{22}^{H}= I_{n}\}$ by Lemma \ref{lem:2.2}. Then by Lemma \ref{lem:2.1} we have
\BE\label{nf:bi}
\max\limits_{\Phi_{2}\in\mathbb{U}_{p}}\tr(B^{H}\Phi_{2}^{H}P_{i}\Phi_{2}B(A^{H}A+B^{H}B)^{-1})=\max\limits_{\psi_{22}\psi_{22}^{H}= I_{n}}\tr(\hat{\Sigma}_{B}^{H}\psi_{22}^{H}F_{i}\psi_{22}\hat{\Sigma}_{B})=\beta_{i}^{2}+\cdots+\beta_{n}^{2}.
\EE
Similarly, we have
\BE\label{nf:bi1}
\max\limits_{\Psi_{1}\in\mathbb{U}_{p}}\tr(B^{H}\Psi_{1}^{H}P_{i+1}\Psi_{1} B(A^{H}A+B^{H}B)^{-1})=\beta_{i+1}^{2}+\cdots+\beta_{n}^{2}.
\EE
Using  \eqref{nf:bi}  and \eqref{nf:bi1} we can conclude that \eqref{def:bi} holds  for $1\leq i\leq n$.

c) If $m\leq n$, then it is easy to check that $\alpha_{m+1}=\cdots=\alpha_{n}=0$. For $1\leq i\leq m$, let $\mathcal{Q}_{i}$ be defined by \eqref{def:cqi}. We set $\Sigma_{A}=(\hat{\Sigma}_{A},O_{m\times(n-m)})$. Using  the GSVD of $\{A,B\}$ in \eqref{gsvd1} and \eqref{gsvd2} and Lemma \ref{lem:2.1} we have 
\begin{align}\label{nf:ai-2}
&\max\limits_{\Psi_{1}\in\mathbb{U}_{m}}\tr(A^{H}\Psi_{1}^{H}\mathcal{Q}_{i}\Psi_{1} A(A^{H}A+B^{H}B)^{-1})\nonumber\\
=&\max\limits_{\Psi_{1}\in\mathbb{U}_{m}}\tr(\Sigma_{A}^{H}U^{H}\Psi_{1}^{H}\mathcal{Q}_{i}\Psi_{1} U\Sigma_{A})\nonumber\\
=&\max\limits_{\Psi_{1}\in\mathbb{U}_{m}}\tr\left(\diag(\hat{\Sigma}_{A}^{H}U^{H}\Psi_{1}^{H}\mathcal{Q}_{i}\Psi_{1} U\hat{\Sigma}_{A}, O_{(n-m)\times (n-m)})\right)\nonumber\\
=&\max\limits_{\Psi_{1}\in\mathbb{U}_{m}}\tr(\hat{\Sigma}_{A}^{H}U^{H}\Psi_{1}^{H}\mathcal{Q}_{i}\Psi_{1} U\hat{\Sigma}_{A})\nonumber\\
=&\alpha_{1}^{2}+\cdots+\alpha_{i}^{2}.
\end{align}
Similarly, we have
\begin{align}\label{nf:ai1-2}
\max\limits_{\Psi_{1}\in\mathbb{U}_{m}}\tr(A^{H}\Psi_{1}^{H}\mathcal{Q}_{i-1}\Psi_{1} A(A^{H}A+B^{H}B)^{-1})=\alpha_{1}^{2}+\cdots+\alpha_{i-1}^{2}.
\end{align}
From \eqref{nf:ai-2}  and \eqref{nf:ai1-2} it follows that \eqref{def:ai-2} holds for $1\leq i\leq m$.

d) If $p\leq n$, then $\beta_{1}=\cdots=\beta_{n-p}=0$.
For $n-p+1\leq i\leq n$, let $\mathcal{P}_{i}$ be defined by \eqref{def:cpi}. We set $\Sigma_{B}=(O_{p\times(n-p)},\hat{\Sigma}_{B})$. From the GSVD of $\{A,B\}$ in \eqref{gsvd1} and \eqref{gsvd2} and Lemma \ref{lem:2.1} we have
\begin{align}\label{nf:bi-2}
&\max\limits_{\Psi_{2}\in\mathbb{U}_{p}}\tr(B^{H}\Psi_{2}^{H}\mathcal{P}_{i}\Psi_{2}B(A^{H}A+B^{H}B)^{-1})\nonumber\\
=& \max\limits_{\Psi_{2}\in\mathbb{U}_{p}}\tr((A^{H}A+B^{H}B)^{-1/2}B^{H}\Psi_{2}^{H}\mathcal{P}_{i}\Psi_{2}B(A^{H}A+B^{H}B)^{-1/2})\nonumber\\
=&\max\limits_{\Psi_{2}\in\mathbb{U}_{p}}\tr(\Sigma_{B}^{H}V^{H}\Psi_{2}^{H}\mathcal{P}_{i}\Psi_{2}V\Sigma_{B})\nonumber\\
=&\max\limits_{\Psi_{2}\in\mathbb{U}_{p}}\tr\left(\diag(O_{(n-p)\times (n-p)}, \hat{\Sigma}_{B}^{H}V^{H}\Psi_{2}^{H}\mathcal{P}_{i}\Psi_{2}V\hat{\Sigma}_{B})\right)\nonumber\\
=&\max\limits_{\Psi_{2}\in\mathbb{U}_{p}}\tr(\hat{\Sigma}_{B}^{H}V^{H}\Psi_{2}^{H}\mathcal{P}_{i}\Psi_{2}V\hat{\Sigma}_{B})\nonumber\\
=&\beta_{i}^{2}+\cdots+\beta_{n}^{2}.
\end{align}
Similarly, we have
\begin{align}\label{nf:bi1-2}
\max\limits_{\Psi_{2}\in\mathbb{U}_{p}}\tr(B^{H}\Psi_{2}^{H}\mathcal{P}_{i+1}\Psi_{2}B(A^{H}A+B^{H}B)^{-1})=\beta_{i+1}^{2}+\cdots+\beta_{n}^{2}.
\end{align}
Using \eqref{nf:bi-2}  and \eqref{nf:bi1-2} we have \eqref{def:bi-2} holds for $n-p+1\leq i\leq n$.
The proof is complete.
\end{proof}

Next, we give new formula model of any GSV of a GMP by trace function optimization under two variables.

\begin{lemma}\label{lem:2.4}
Let $\{A,B\}$ be an (m,p,n)-GMP and the GSV $(\alpha_{i},\beta_{i})$ of $\{A,B\}$
be given  in \eqref{gsvd1} and \eqref{gsvd2}.  The following conclusions hold true. 
\begin{itemize}
\item[\rm a)]
If $n\leq m,$ then for $1\leq i\leq n$,
\begin{eqnarray}\label{pi:ai}
\alpha_{i} &=&\max\limits_{\Pi_{1}\in\mathbb{U}_{m},\Pi_{2}\in\mathbb{U}_{n}}|\tr(\Pi_{1}A(A^{H}A+B^{H}B)^{-1/2}\Pi_{2}\mathcal{G}_{i})| \nonumber\\
&&- \max\limits_{\Pi_{1}\in\mathbb{U}_{m},\Pi_{2}\in\mathbb{U}_{n}}|\tr(\Pi_{1}A(A^{H}A+B^{H}B)^{-1/2}\Pi_{2}\mathcal{G}_{i-1})|\equiv\tilde{\phi}_{i},
\end{eqnarray}
where
\BE\label{def:Gi}
\mathcal{G}_{i}=(G_{i},O_{n\times(m-n)}),\quad G_{i}=\diag(I_{i},O_{(n-i)\times (n-i)}).
\EE
\item[\rm b)] If $n\leq p$, then for $1\leq i\leq n$,
\begin{eqnarray}\label{pi:bi}
\beta_{i} &=& \max\limits_{\Xi_{1}\in\mathbb{U}_{p},\Xi_{2}\in\mathbb{U}_{n}}|\tr(\Xi_{1}B(A^{H}A+B^{H}B)^{-1/2}\Xi_{2}\mathcal{H}_{i})| \nonumber\\
&&-\max\limits_{\Xi_{1}\in\mathbb{U}_{p},\Xi_{2}\in\mathbb{U}_{n}}|\tr(\Xi_{1}B(A^{H}A+B^{H}B)^{-1/2}\Xi_{2}\mathcal{H}_{i-1})|\equiv\tilde{\psi}_{i},
\end{eqnarray}
where
\BE\label{def:Hi}
\mathcal{H}_{i}=(O_{n\times (p-n)},H_{i}),\quad H_{i}=\diag(O_{(i-1)\times (i-1)},I_{n-i+1}).
\EE
\item[\rm c)]
If $m\leq n$, then $\alpha_{m+1}=\cdots=\alpha_{n}=0$ and for $1\leq i\leq m$,
\begin{eqnarray}\label{pi:ai-2}
 \alpha_{i}&=&\max\limits_{\Pi_{3}\in\mathbb{U}_{n},\Pi_{4}\in\mathbb{U}_{m}}|\tr(\Pi_{3}(A^{H}A+B^{H}B)^{-1/2}A^{H}\Pi_{4}\mathcal{S}_{i})| \nonumber\\
 &&-\max\limits_{\Pi_{3}\in\mathbb{U}_{n},\Pi_{4}\in\mathbb{U}_{m}}|\tr(\Pi_{3}(A^{H}A+B^{H}B)^{-1/2}A^{H}\Pi_{4}\mathcal{S}_{i-1})|\equiv\tilde{\varphi}_{i},
\end{eqnarray}
\BE\label{def:Si}
\mathcal{S}_{i}=(S_{i},O_{m\times(n-m)}),\quad S_{i}=\diag(I_{i},O_{(m-i)\times (m-i)}).
\EE

\item[\rm d)]
If $p\leq n$, then $\beta_{1}=\cdots=\beta_{n-p}=0$ and for $n-p+1\leq i\leq n$,
\begin{eqnarray}\label{pi:bi-2}
\beta_{i}&=&\max\limits_{\Xi_{3}\in\mathbb{U}_{n},\Xi_{4}\in\mathbb{U}_{p}}|\tr(\Xi_{3}(A^{H}A+B^{H}B)^{-1/2}B^{H}\Xi_{4}\mathcal{W}_{i})| \nonumber\\
&&- \max\limits_{\Xi_{3}\in\mathbb{U}_{n},\Xi_{4}\in\mathbb{U}_{p}}|\tr(\Xi_{3}(A^{H}A+B^{H}B)^{-1/2}B^{H}\Xi_{4}\mathcal{W}_{i-1})|\equiv\tilde{\chi}_{i},
\end{eqnarray}
where
\BE\label{def:Wi}
\mathcal{W}_{i}=(O_{p\times (n-p)},W_{i}), \quad W_{i}=\diag(O_{(p-n+i-1)\times (p-n+i-1)},I_{n-i+1}).
\EE
\end{itemize}
\end{lemma}
\begin{proof} a) If $n\leq m,$ then for $1\leq i\leq n$, we let $\mathcal{G}_{i}$ be defined by \eqref{def:Gi}. We set $\Sigma_{A}= \begin{pmatrix}\hat{\Sigma}_{A}\\O_{(m-n)\times n}\end{pmatrix}$. For any $\Pi_{1}\in\mathbb{U}_{m}$, let $\Pi_{1}U=\begin{pmatrix}\Pi_{11} &\Pi_{12}\\ \Pi_{21}&\Pi_{22}\end{pmatrix}$ with $\Pi_{11}\in\mathbb{C}^{n\times n}$.  By the GSVD of $\{A,B\}$ in \eqref{gsvd1} and \eqref{gsvd2} we have
 \begin{align*} &\max\limits_{\Pi_{1}\in\mathbb{U}_{m},\Pi_{2}\in\mathbb{U}_{n}}|\tr(\Pi_{1}A(A^{H}A+B^{H}B)^{-1/2}\Pi_{2}\mathcal{G}_{i})|\\ =&\max\limits_{\Pi_{1}\in\mathbb{U}_{m},\Pi_{2}\in\mathbb{U}_{n}}|\tr(\Pi_{1}U\Sigma_{A}R(R^{H}R)^{-1/2}\Pi_{2}\mathcal{G}_{i})|\\ =&\max\limits_{\Pi_{11}\Pi_{11}^{H}\leq I_{n},\Pi_{2}\in\mathbb{U}_{n}}|\tr(\Pi_{11}\hat{\Sigma}_{A}R(R^{H}R)^{-1/2}\Pi_{2}G_{i})|.
 \end{align*}
 Observe $f(\Pi_{11})=\tr(\Pi_{11}\hat{\Sigma}_{A}R(R^{H}R)^{-1/2}\Pi_{2}G_{i})$ is an analytical function of several complex variables on the domain $\Pi_{11}\Pi_{11}^{H}\leq I_{n}$. By Lemma \ref{lem:2.2} we know that $f(\Pi_{11})$ attains its maximum modulus on the characteristic manifold $\{\Pi_{11}\in\mathbb{C}^{n\times n}:\Pi_{11}\Pi_{11}^{H}= I_{n}\}$. Therefore, 
\begin{eqnarray*} 
&& \max\limits_{\Pi_{11}\Pi_{11}^{H}\leq I_{n},\Pi_{2}\in\mathbb{U}_{n}}|\tr(\Pi_{11}\hat{\Sigma}_{A}R(R^{H}R)^{-1/2}\Pi_{2}G_{i})| \\ &=&\max\limits_{\Pi_{11}\Pi_{11}^{H}= I_{n},\Pi_{2}\in\mathbb{U}_{n}}|\tr(\Pi_{11}\hat{\Sigma}_{A}R(R^{H}R)^{-1/2}\Pi_{2}G_{i})| 
 \end{eqnarray*}
 and it follows from Lemma \ref{lem:2.1} that
\begin{eqnarray} \label{nf:ai-pi}
&&\max\limits_{\Pi_{1}\in\mathbb{U}_{m},\Pi_{2}\in\mathbb{U}_{n}}|\tr(\Pi_{1}A(A^{H}A+B^{H}B)^{-1/2}\Pi_{2}\mathcal{G}_{i})| 
\nonumber\\ 
&&=\max\limits_{\Pi_{11}\Pi_{11}^{H}= I_{n},\Pi_{2}\in\mathbb{U}_{n}}|\tr(\Pi_{11}\hat{\Sigma}_{A}R(R^{H}R)^{-1/2}\Pi_{2}G_{i})|\nonumber\\
 &&=\alpha_{1}+\cdots+\alpha_{i}.
\end{eqnarray}
 Similarly, we have
\begin{eqnarray} \label{nf:ai1-pi}
 \max\limits_{\Pi_{1}\in\mathbb{U}_{m},\Pi_{2}\in\mathbb{U}_{n}}|\tr(\Pi_{1}A(A^{H}A+B^{H}B)^{-1/2}\Pi_{2}\mathcal{G}_{i-1})|=\alpha_{1}+\cdots+\alpha_{i-1}. 
 \end{eqnarray}
From \eqref{nf:ai-pi} and \eqref{nf:ai1-pi}  we know that \eqref{pi:ai} holds for $1\le i\le n$.
 
b) If $n\leq p$, then for  $1\le i\le n$, let $\mathcal{H}_{i}$ be defined by \eqref{def:Hi}. We set $\Sigma_{B}= \begin{pmatrix}O_{(p-n)\times n}\\ \hat{\Sigma}_{B}\end{pmatrix}$. For any $\Xi_{1}\in\mathbb{U}_{p}$, let $\Xi_{1}V=\begin{pmatrix}\Xi_{11} & \Xi_{12}\\ \Xi_{21}&\Xi_{22}\end{pmatrix}\in\mathbb{U}_{p}$ with $\Xi_{11}\in\mathbb{C}^{n\times n}$. By  the GSVD of $\{A,B\}$ in \eqref{gsvd1} and \eqref{gsvd2},  Lemma \ref{lem:2.2},  and Lemma \ref{lem:2.1} we have 
 \begin{eqnarray} \label{nf:bi-pi}
 &&\max\limits_{\Xi_{1}\in\mathbb{U}_{p},\Xi_{2}\in\mathbb{U}_{n}}|\tr(\Xi_{1}B(A^{H}A+B^{H}B)^{-1/2}\Xi_{2}\mathcal{H}_{i})|\nonumber\\  
 &-& \max\limits_{\Xi_{1}\in\mathbb{U}_{p},\Xi_{2}\in\mathbb{U}_{n}}|\tr(\Xi_{1}V\Sigma_BR(R^{H}R)^{-1/2}\Xi_{2}\mathcal{H}_{i})|\nonumber\\  
 &=&\max\limits_{\Xi_{22}\Xi_{22}^{H}\leq I_{n},\Xi_{2}\in\mathbb{U}_{n}}|\tr(\Xi_{22}\hat{\Sigma}_{B}R(R^{H}R)^{-1/2}\Xi_{2}H_{i})|\nonumber\\
 &=&\max\limits_{\Xi_{22}\Xi_{22}^{H}= I_{n},\Xi_{2}\in\mathbb{U}_{n}}|\tr(\Xi_{22}\hat{\Sigma}_{B}R(R^{H}R)^{-1/2}\Xi_{2}H_{i})|\nonumber\\ &=&\beta_{i}+\cdots+\beta_{n}.
 \end{eqnarray}
 Similarly, we have
 \begin{align}\label{nf:bi1-pi}
 \max\limits_{\Xi_{1}\in\mathbb{U}_{p},\Xi_{2}\in\mathbb{U}_{n}}|\tr(\Xi_{1}B(A^{H}A+B^{H}B)^{-1/2}\Xi_{2}\mathcal{H}_{i+1})| =\beta_{i+1}+\cdots+\beta_{n}.
 \end{align}
Using  \eqref{nf:bi-pi} and \eqref{nf:bi1-pi}  we know that \eqref{pi:bi} holds for $1\le i\le n$.

c) If $m\leq n$, then  it is easy to check that $\alpha_{m+1}=\cdots=\alpha_{n}=0$. For $1\leq i\leq m$, let $\mathcal{S}_{i}$ be defined by \eqref{def:Si}. We set $\Sigma_{A}=(\hat{\Sigma}_{A},O_{m\times(n-m)})$ and $\Pi_{3}(R^{H}R)^{-1/2}R^{H}=\begin{pmatrix}\tilde{\Pi}_{11}&\tilde{\Pi}_{12}\\ \tilde{\Pi}_{21}&\tilde{\Pi}_{22}\end{pmatrix}\in\mathbb{U}_{n}$.  By  the GSVD of $\{A,B\}$ in \eqref{gsvd1} and \eqref{gsvd2}, Lemma \ref{lem:2.2}, and Lemma  \ref{lem:2.1}  we have
 \begin{eqnarray} \label{nf:ai-pi2}
 &&\max\limits_{\Pi_{3}\in\mathbb{U}_{n},\Pi_{4}\in\mathbb{U}_{m}}|\tr(\Pi_{3}(A^{H}A+B^{H}B)^{-1/2}A^{H}\Pi_{4}\mathcal{S}_{i})|\nonumber\\
 &=& \max\limits_{\Pi_{3}\in\mathbb{U}_{n},\Pi_{4}\in\mathbb{U}_{m}}|\tr(\Pi_{3}(R^{H}R)^{-1/2}R^{H}\Sigma_A^HU^H\Pi_{4}\mathcal{S}_{i})|\nonumber\\
&=&\max\limits_{\tilde{\Pi}_{11}\tilde{\Pi}_{11}^{H}\leq I_{m},\Pi_{4}\in\mathbb{U}_{m}}|\tr(\tilde{\Pi}_{11}\hat{\Sigma}_{A}^{H}U^{H}\Pi_{4}S_{i})|\nonumber\\
&=&\max\limits_{\tilde{\Pi}_{11}\tilde{\Pi}_{11}^{H}= I_{m},\Pi_{4}\in\mathbb{U}_{m}}|\tr(\tilde{\Pi}_{11}\hat{\Sigma}_{A}^{H}U^{H}\Pi_{4}S_{i})|\nonumber\\
&=&\alpha_{1}+\cdots+\alpha_{i}.
 \end{eqnarray}
 Similarly, we have
 \begin{align}\label{nf:ai1-pi2}
 \max\limits_{\Pi_{3}\in\mathbb{U}_{n},\Pi_{4}\in\mathbb{U}_{m}}|\tr(\Pi_{3}(A^{H}A+B^{H}B)^{-1/2}A^{H}\Pi_{4}\mathcal{S}_{i-1})| =\alpha_{1}+\cdots+\alpha_{i-1}.
 \end{align}
From  \eqref{nf:ai-pi2} and \eqref{nf:ai1-pi2}  we know that \eqref{pi:ai-2} holds for $1\le i\le m$.

d) If $p\leq n$, then we know that $\beta_{1}=\cdots=\beta_{n-p}=0$. For  $n-p+1\leq i\leq n$, let $\mathcal{W}_{i}$ be defined by \eqref{def:Wi}. We set $\Sigma_{B}=(O_{p\times(n-p)},\hat{\Sigma}_{B})$ and $\Xi_{3}(R^{H}R)^{-1/2}R^{H}=\begin{pmatrix}\tilde{\Xi}_{11}&\tilde{\Xi}_{12}\\ \tilde{\Xi}_{21}&\tilde{\Xi}_{22}\end{pmatrix}\in\mathbb{U}_{n}$.  By  the GSVD of $\{A,B\}$ in \eqref{gsvd1} and \eqref{gsvd2}, Lemma \ref{lem:2.2}, and Lemma  \ref{lem:2.1}  we have
 \begin{eqnarray} \label{nf:bi-pi2}
&&\max\limits_{\Xi_{3}\in\mathbb{U}_{n},\Xi_{4}\in\mathbb{U}_{p}}|\tr(\Xi_{3}(A^{H}A+B^{H}B)^{-1/2}B^{H}\Xi_{4}\mathcal{W}_{i})|\nonumber\\
&=&\max\limits_{\Xi_{3}\in\mathbb{U}_{n},\Xi_{4}\in\mathbb{U}_{p}}|\tr(\Xi_{3}(R^{H}R)^{-1/2}R^{H}\Sigma_{B}^{H}V^{H}\Xi_{4}\mathcal{W}_{i})|\nonumber\\
&=&\max\limits_{\tilde{\Xi}_{22}\tilde{\Xi}_{22}^{H}\leq I_{p},\Xi_{4}\in\mathbb{U}_{p}}|\tr(\tilde{\Xi}_{22}\hat{\Sigma}_{B}^{H}V^{H}\Xi_{4}W_{i})|\nonumber\\
&=&\max\limits_{\tilde{\Xi}_{22}\tilde{\Xi}_{22}^{H}= I_{p},\Xi_{4}\in\mathbb{U}_{p}}|\tr(\tilde{\Xi}_{22}\hat{\Sigma}_{B}^{H}V^{H}\Xi_{4}W_{i})|\nonumber\\
&=&\beta_{i}+\cdots+\beta_{n}.
\end{eqnarray}
 Similarly, we have
 \begin{eqnarray} \label{nf:bi1-pi2}
 \max\limits_{\Xi_{3}\in\mathbb{U}_{n},\Xi_{4}\in\mathbb{U}_{p}}|\tr(\Xi_{3}(A^{H}A+B^{H}B)^{-1/2}B^{H}\Xi_{4}\mathcal{W}_{i+1})| =\beta_{i+1}+\cdots+\beta_{n}.
 \end{eqnarray}
Using  \eqref{nf:bi-pi2} and \eqref{nf:bi1-pi2}  we can conclude that \eqref{pi:bi-2} holds for $n-p+1\leq i\leq n$. 
The proof is complete.
\end{proof}

Based on Lemmas \ref{lem:2.3} and \ref{lem:2.4}, we give new model formulations of any GSV $(\alpha_{i},\beta_{i})$ of the  $(m,p,n)$-GMP $\{A,B\}$.  We note that $\alpha_{i}^{2}+\beta_{i}^{2}=1$ for $1\le i\le n$  and
for $m\leq n$, $\alpha_{m+1}=\cdots=\alpha_{n}=0
$ and for $p\leq n$, $\beta_{1}=\cdots=\alpha_{n-p}=0$. Then, for $1\leq i\leq [n/2]$ we use
$\alpha_{i}$'s formulations and for $[n/2]+1\leq i\leq n$ we use $\beta_{i}$'s formulations. Therefore, we can derive the following results.
\begin{theorem}\label{thm:2.5}
Let $\{A,B\}$ be an $(m,p,n)$-GMP and the GSV $(\alpha_{i},\beta_{i})$ of $\{A,B\}$
be given in  \eqref{gsvd1} and \eqref{gsvd2}. The following conclusions hold true. 
\begin{itemize}
\item[\rm a)]
If $n\leq m$, then we use
\begin{eqnarray}
\alpha_{i}=\phi_{i}^{\frac{1}{2}}(\tilde{\phi}_{i}),\;\mbox{for}\; 1\leq i\leq [n/2].\nonumber
\end{eqnarray}
\item[\rm b)] If $m<n$ and further if $m<[n/2]$, then we use
 \begin{eqnarray}
\alpha_{i}=\varphi_{i}^{\frac{1}{2}}(\tilde{\varphi}_{i}),\;for\;1\leq i\leq m,\;\alpha_{m+1}=\cdots=\alpha_{[n/2]}=0;\nonumber
\end{eqnarray}
If $m<n$ and further if $[n/2]\leq m$, then we use
 \begin{eqnarray}
  \alpha_{i}=\varphi_{i}^{\frac{1}{2}}(\tilde{\varphi}_{i}),\;\mbox{for}\; 1\leq i\leq[n/2].\nonumber
\end{eqnarray}
\item[\rm c)] If $n\leq p$, then we use
\begin{eqnarray}
\beta_{i}= \psi_{i}^{\frac{1}{2}}(\tilde{\psi}_{i}),\;\mbox{for}\;[n/2]+1\leq i\leq n.\nonumber
\end{eqnarray}
\item[\rm d)] If $p< n$ and further if $n-p<[n/2]$, then we use
 \begin{eqnarray}
\beta_{i}=\chi_{i}^{\frac{1}{2}}(\tilde{\chi}_{i}),\;\mbox{for}\;[n/2]+1\leq i\leq n;\nonumber
\end{eqnarray}
If $p< n$ and further if $n-p>[n/2]$, then we use
 \begin{eqnarray}
\beta_{[n/2]+1}=\cdots=\beta_{n-p}=0,\;\beta_{i}=\chi_{i}^{\frac{1}{2}}(\tilde{\chi}_{i}),\;\mbox{for}\;n-p+1\leq i\leq n.\nonumber
\end{eqnarray}
\end{itemize}
\end{theorem}
\begin{remark} \label{rem:2.1}
By Lemma \ref{lem:2.3}  we can provide new formula model of any GSV $(\alpha_{i},\beta_{i})$ of the GMP $\{A,B\}$ by trace function under one variable. By Lemma \ref{lem:2.4}, we can provide  new formula model of any GSV $(\alpha_{i},\beta_{i})$ of the GMP $\{A,B\}$ by trace function under two variables. By Theorem \ref{thm:2.5} and Lemma  \ref{lem:2.3}  we know that for $1\leq i\leq [n/2]$, we use
$\alpha_{i}$'s formula referring to $Q_{i},\mathcal{Q}_{i}$ involving more zeros (about $n-i$ or $m-i$ zeros) as their diagonal entries and for $[n/2]+1\leq i\leq n$, we use $\beta_{i}$'s formula
referring to $P_{i},\mathcal{P}_{i}$ involving more zeros (about $i$ zeros) as their diagonal entries by trace function under one variable. By Theorem \ref{thm:2.5}  and Lemma  \ref{lem:2.4}  we know that for $1\leq i\leq [n/2]$, we use $\alpha_{i}$'s formula referring to $\mathcal{G}_{i},\mathcal{S}_{i}$ involving more zeros (about $n-i,m-i$ zeros) as their diagonal entries and for $[n/2]+1\leq i\leq n$, we use $\beta_{i}$'s formula referring to $\mathcal{H}_{i},\mathcal{W}_{i}$ involving more zeros (about $i$ zeros) as their diagonal entries by trace function under two variables.
\section{Numerical methods for computing $\alpha_{i}$ and $\beta_{i}$}
In this section, based on trace function optimization under one variable and two variables, we give computational methods for computing $\alpha_{i}$ and $\beta_{i}$, respectively.
\end{remark}

\subsection{Based on trace function optimization under one variable}
In this subsection, based on trace function optimization under one variable, we present Newton's method on Grassmann manifold for computing $\alpha_{i}$ and $\beta_{i}$ by Theorem \ref{thm:2.5} and Lemma  \ref{lem:2.3}.

Let $\{A,B\}$ be an $(m,p,n)$-GMP. For demonstration purpose, we assume that $n\leq m$ and $n\le p$. To compute $\alpha_{i}^2$ as defined in Theorem \ref{thm:2.5}, we consider the following matrix optimization problems:
\BE\label{f1}
\begin{array}{cc}
\max & \displaystyle f_1(\Phi_1):=\frac{1}{2}\tr(A^{H}\Phi_{1}^{H}Q_{i}\Phi_{1} A(A^{H}A+B^{H}B)^{-1}) \\[2mm]
{\rm s.t.} & \Phi_{1}\in\mathbb{U}_{m}
\end{array}
\EE
and
\BE\label{f2}
\begin{array}{cc}
\max & \displaystyle f_2(\Phi_2):=\frac{1}{2}\tr(B^{H}\Phi_{2}^{H}P_{i}\Phi_{2} B(A^{H}A+B^{H}B)^{-1})\\[2mm]
{\rm s.t.} & \Phi_{2}\in\mathbb{U}_{p}.
\end{array}
\EE
By using the definitions of $Q_i$ and $P_i$ and $\alpha_{i}^{2}+\beta_{i}^{2}=1$ for $1\leq i\leq n$, we only need to focus on the following matrix optimization problems:
\BE\label{f1-eq}
\begin{array}{cc}
\max & \displaystyle f_1(\Phi_{11}):=\frac{1}{2}\tr(\Phi_{11}^{H}C\Phi_{11}) \\[2mm]
{\rm s.t.} & \Phi_{11}\in\st(m,i):=\{\Phi_{11}\in \mathbb{C}^{m\times i}\;|\; \Phi_{11}^H\Phi_{11}=I_{i}\},
\end{array}
\EE

and
\BE\label{f2-eq}
\begin{array}{cc}
\max & \displaystyle f_2(\Phi_{22}):=\frac{1}{2}\tr(\Phi_{22}^{H}D\Phi_{22}) \\[2mm]
{\rm s.t.} & \Phi_{22}\in\st(p,n-i+1),
\end{array}
\EE
for $1\le i\le[n/2]$,
where $C=A(A^{H}A+B^{H}B)^{-1})A^H$, $D=B(A^{H}A+B^{H}B)^{-1})B^H$, $\Phi_{11}$ is the first $i$ columns of $\Phi_{1}^H$ and $\Phi_{22}$ is the last $n-i+1$ columns of $\Phi_{2}^H$.

In the following, we focus on the solution of the matrix optimization problem (\ref{f1-eq}). The matrix optimization  problem (\ref{f2-eq}) can be solved  in a similar way.

We note that $f_1(\Phi_{11}Q)=f_1(\Phi_{11})$ for all $\Phi_{11}\in\st(m,i)$ and $Q\in\mathbb{U}_{i}$. As in \cite{M02}, let the complex Grassmann manifold $\grass(m,i)$ be the set of all $i$-dimensional complex subspace of $\mathbb{C}^m$. If $[X]$ means the subspace spanned by the columns of $X\in\st(m,i)$, then we have $[X]\in\grass(m,i)$. In particular, for any $X\in\st(m,i)$,  the natural projection $[X]\in\grass(m,i)$ corresponds  the equivalent class $\{XQ\in\st(m,i) \ | \ Q\in \mathbb{U}_i\}$ of $\st(m,i)$. Thus, instead of  problem (\ref{f1-eq}), we consider the following optimization problem:
\BE\label{f1-gr}
\begin{array}{cc}
\max & \displaystyle \tilde{f}_1([\Phi_{11}]):= f_1(\Phi_{11}) \\[2mm]
{\rm s.t.} & [\Phi_{11}] \in\grass(m,i).
\end{array}
\EE

Next, we present Newton's method for solving the optimization problem (\ref{f1-gr}).
Let $\Phi_{11}\in\st(m,i)$. The tangent space of $\st(m,i)$ at $\Phi_{11}$ is given by \cite{M02}
\[
 T_{\Phi_{11}}\st(m,i)=\{Z\in\mathbb{C}^{m\times i} \; | \; Z=\Phi_{11}\Omega + (\Phi_{11})_\perp K, \Omega^H=-\Omega, \Omega\in\mathbb{C}^{i,i}, K\in\mathbb{C}^{m-i,i}\},
\]
which can be endowed with the inner product
\[
\langle Z_1,Z_2\rangle  =\mathfrak{R}[\tr(Z_2^H(I-\frac{1}{2}\Phi_{11}\Phi_{11})^H)Z_1)],\quad Z_1,Z_2\in T_{\Phi_{11}}\st(m,i), \Phi_{11}\in\st(m,i).
\]
Here, $(\Phi_{11})_\perp$ means  that span$((\Phi_{11})_\perp)$ is the orthogonal complement of span$(\Phi_{11})$, where span$(\Phi_{11})$ denotes a linear space spanned by the column vectors of $\Phi_{11}$.
We note that the tangent space of $\grass(m,i)$ at $[\Phi_{11}]$ is given by \cite{M02}
 \[
 T_{[\Phi_{11}]}\grass(m,i)=\{Z\in\mathbb{C}^{m\times i} \; | \; Z=(\Phi_{11})_{\bot}K, K\in\mathbb{C}^{m-i,i}\}
 \subset T_{\Phi_{11}}\st(m,i).
 \]
 Hence, we can define a Riemannian metric on $\grass(m,i)$ by
 \[
 \langle Z_1,Z_2\rangle  =\mathfrak{R}[\tr(Z_2^HZ_1)],\quad Z_1,Z_2\in T_{[\Phi_{11}]}\grass(m,i), \Phi_{11}\in\st(m,i)
 \]
 with the induced norm $\|\cdot\|$.  Then the orthogonal projection onto $T_{[\Phi_{11}]}\grass(m,i)$ is given by
 \[
 P_{\Phi_{11}}Z=(I-\Phi_{11}\Phi_{11}^H)Z,\quad \forall Z\in\mathbb{C}^{m\times i}.
 \]
 We define the local cost function $g:T_{[\Phi_{11}]}\grass(m,i)\to\mathbb{R}$ by
 \[
 g(Z)= \tilde{f}_1([\Phi_{11}+Z]).
 \]
 It is easy to check that, for any $Z\in T_{[\Phi_{11}]}\grass(m,i)$,
 \begin{eqnarray*}
 g(Z)&=& f_1(\Phi_{11}) +\mathfrak{R}\tr(Z^{H}D_{\Phi_{11}}) \\
 &&+\frac{1}{2}\ve(Z)^{H}\Big(H_{\Phi_{11}} - (\Phi_{11}^HD_{\Phi_{11}} )^T\otimes I_m\Big)\ve(Z),
\end{eqnarray*}
where $D_{\Phi_{11}}=C\Phi_{11}$ and $H_{\Phi_{11}}=I_{i}\otimes C$ are the derivative and Hessian of $f_1$ at $\Phi_{11}$, respectively.

Based on the above analysis, Newton's method for solving the optimization problem (\ref{f1-gr}) can be described as follows \cite{M02}.
\begin{algorithm}  \label{nm}
{\rm }
\begin{description}
\item [{\rm Step 0.}] Choose $\Phi_{11}^0\in\st(m,i)$, $\rho,\eta\in(0,1)$, $\sigma\in (0,1/2]$ and let $k:=1$.
\item [{\rm Step 1.}] Apply the conjugate gradient (CG) method \cite[Chap. 11.3]{1} to solving
\[
P_{\Phi_{11}^k}(CZ^k - Z^k(\Phi_{11}^k)^HD_{\Phi_{11}^k}) = -P_{\Phi_{11}^k}D_{\Phi_{11}^k}
\]
for $Z^k\in T_{[\Phi_{11}^k]}\grass(m,i)$ such that
\BE\label{cnd1}
\|P_{\Phi_{11}^k}(CZ^k - Z^k(\Phi_{11}^k)^HD_{\Phi_{11}^k}) +P_{\Phi_{11}^k}D_{\Phi_{11}^k}\|\le \eta_k \|P_{\Phi_{11}^k}D_{\Phi_{11}^k}\|
\EE
and
\BE\label{cnd2}
\langle P_{\Phi_{11}^k}D_{\Phi_{11}^k}, Z^k\rangle \le\eta_k \langle Z^k, Z^k\rangle,
\EE
where $\eta_k=\min\{\eta, \|P_{\Phi_{11}^k}D_{\Phi_{11}^k}\| \}$. If (\ref{cnd1}) and  (\ref{cnd2}) are not attainable, then let
\[
Z^k=-P_{\Phi_{11}^k}D_{\Phi_{11}^k}.
\]
\item [{\rm Step 2.}] Let $l_k>0$ be the smallest integer $m$ such that
\[
f_1(\pi(\Phi_{11}^k+\beta^lZ^k))\le f_1(\Phi_{11}^k) +\sigma\rho^l \langle P_{\Phi_{11}^k}D_{\Phi_{11}^k}, Z^k\rangle
\]
Set
\[
Z^{k+1}=\pi({\Phi_{11}^k}+\beta^lZ^k).
\]
\item [{\rm Step 3.}] Replace $k$ by $k+1$ and go to Step 1.
\end{description}
\end{algorithm}

In Step 2 of Algorithm \ref{nm}, $\pi:\C^{m\times i}\to\grass(m,i)$ is the projection onto $\grass(m,i)$, which is  defined as follows: Let $X\in\C^{m\times i}$ be full column rank. Then
\[
\pi(X)=\left[\argmin_{Y\in\st(m,i)}\|X-Y\|^2\right].
\]
As noted in \cite{M02}, if the SVD of $X$ is given by $X=U\Sigma V^H$, then $\pi(X)=UI_{m,i}V^H$. If the QR decomposition of $X$ is $X=QR$, then $\pi(X)=Q_{m,i}$.

We must point out that, if $A$ and $B$ are real matrices, then the matrix $C$ in problem (\ref{f1-eq}) is real. In this case, one may use the trust-region methods in \cite{AM08} or the geometric Newton algorithm in \cite{ZB15} for solving the  optimization problem (\ref{f1-eq}).
\subsection{Based on trace function optimization under two variables}

In the following, based on trace function optimization under two variables, we use classical Golub-Kahan bidiagonalization method (which is implemented by {\tt MATLAB}-routine function {\tt svds}) for computing $\alpha_{i}$ and $\beta_{i}$ by Theorem \ref{thm:2.5}  and Lemma  \ref{lem:2.4}. For simplicity, we assume that $n\leq m$ and $n\le p$. To compute $\alpha_{i}$ and $\beta_{i}$ as defined in Lemma \ref{lem:2.4}, we study the following matrix optimization problems:
\BE\label{g1}
\begin{array}{cc}
\max & \displaystyle g_1(\Pi_1,\Pi_2):=|\tr(\Pi_1A(A^{H}A+B^{H}B)^{-1/2}\Pi_2\cg_i)| \\[2mm]
{\rm s.t.} & \Pi_{1}\in\mathbb{U}_{m},\Pi_{2}\in\mathbb{U}_{n}
\end{array}
\EE
and
\BE\label{g2}
\begin{array}{cc}
\max & \displaystyle g_2(\Xi_1,\Xi_2):=|\tr(\Xi_1B(A^{H}A+B^{H}B)^{-1/2}\Xi_2\ch_i)| \\[2mm]
{\rm s.t.} & \Xi_{1}\in\mathbb{U}_{p},  \Xi_{2}\in\mathbb{U}_{n}.
\end{array}
\EE
By using the definitions of $\cg_i$ and $\ch_i$ and $\alpha_{i}^{2}+\beta_{i}^{2}=1$ for $1\leq i\leq n$, we only need to focus on the following matrix optimization problems:
\BE\label{g1-eq}
\begin{array}{cc}
\max & \displaystyle g_1(\Pi_{11},\Pi_{21}):=|\tr(\Pi_{11}^{H}E\Pi_{21})| \\[2mm]
{\rm s.t.} & \Pi_{11}\in\st(m,i), \Pi_{21}\in\st(n,i)
\end{array}
\EE
and
\BE\label{g2-eq}
\begin{array}{cc}
\max & \displaystyle g_2(\Xi_{12},\Xi_{22}):=|\tr(\Xi_{12}^{H}F\Xi_{22})| \\[2mm]
{\rm s.t.} & \Phi_{22}\in\st(p,n-i+1),
\end{array}
\EE
for $1\le i\le [n/2]$,
where $E=A(A^{H}A+B^{H}B)^{-1/2}$, $F=B(A^{H}A+B^{H}B)^{-1/2}$, $\Pi_{11}^H$ is the first $i$ rows of $\Phi_{1}$, $\Pi_{21}$ is the first $i$ columns of $\Pi_{2}$, $\Xi_{12}^H$ is the last $n-i+1$ rows of $\Xi_{1}$, and $\Xi_{22}$ is the last $n-i+1$ columns of $\Xi_{2}$.
We will discuss the solution of the matrix optimization problem (\ref{g1-eq}). The matrix optimization  problem (\ref{g2-eq}) can be solved similarly.

As in \cite[Chap.10.4]{1}, starting a unit $2$-norm vector $\bv_1\in\Cn$, the Golub-Kahan process \cite{GK65} for bidiagonalizing the matrix $E$ can be stated as follows (see also \cite[Algorithm 10.4.1]{1}).
\begin{algorithm}  \label{svd}
{\rm }
\begin{description}
\item [{\rm Step 0.}] Choose $\bv_1\in\C^n$ with $\|\bv_1\|_2=1$, $\bp_0=\bv_1$, $\beta_0=1$, $\bu_0={\bf 0}$, and let $k:=0$.
\item [{\rm Step 1.}] While $\beta_k\neq 0$
 \begin{eqnarray*}
&&\bv_{k+1} = \bp_k/\beta_k\\
&&\br_{k+1} = A\bv_{k+1}-\beta_{k}\bu_{k}\\
&&\alpha_{k+1}=\|\br_{k+1}\|_2\\
&&\bu_{k+1}=\br_{k+1}/\alpha_{k+1}\\
&&\bp_{k+1}=A^H\bu_{k+1} - \alpha_{k+1}\bv_{k+1}\\
&&\beta_{k+1}=\|\bp_{k+1}\|_2
\end{eqnarray*}
end
\item [{\rm Step 2.}] Replace $k$ by $k+1$ and go to Step 1.
\end{description}
\end{algorithm}

From Algorithm \ref{svd}, we can obtain $V_{k}=[\bv_1,\bv_2,\ldots,\bv_k]$, $U_{k}=[\bu_1,\bu_2,\ldots,\bu_k]$ with orthonormal columns and a $k\times k$ bidiagonal matrix
\[
B_k=\left[
\begin{array}{ccccc}
\alpha_1 & \beta_1 & \cdots & \cdots & 0\\
0 & \alpha_2 & \beta_2  &\cdots & \vdots\\
\vdots & \ddots &  \ddots &  \ddots & 0\\
\vdots & &  0 & \alpha_{k-1} & \beta_{k-1}\\
0 & \cdots & 0 &  0 & \alpha_{k}
\end{array}
\right]
\]
such that
\BE\label{eq:bid}
EV_{k} =U_{k}B_{k}\quad\mbox{and}\quad E^HU_k=V_kB^H + \bp_{k}\bfe_{k}^T.
\EE
 Then we can compute the SVD of $B_{k}$ via the SVD algorithm in \cite[Chap. 8.6]{1}
\[
F_{k}^HB_{k}G_{k}=\Gamma_{k}=\diag(\gamma_1,\ldots,\gamma_{k})
\]
and form the matrices
\[
Y_k=V_kG_k\quad\mbox{and}\quad Z_k=U_kF_k.
\]
From (\ref{eq:bid}) we have
\[
EY_{k} =Z_{k}\Gamma_{k}\quad\mbox{and}\quad E^HZ_k=Y_k\Gamma + \bp_{k}\bfe_{k}^TF_k.
\]
Hence, the singular vector matrix $\Pi_{11}$ and $\Pi_{21}$  can be estimated from the $i$ column vectors of $Z_k$ and $Y_k$ corresponding to the $i$ largest singular values of $B_k$.
We then use  {\tt MATLAB}-routine function {\tt svds}. Commands '{tt svds}(A,i)' and '{\tt svds}(A,i,'smallest')' denote $i$ largest and smallest singular values of matrix $A$, respectively.

\begin{remark} \label{rem:3.1}
The given model formulas by trace function optimization under one variable and two variables  are new.
Based on the new formulas, we give computational methods involving  Newton's method on Grassmann manifold and
Algorithm 10.4.1 in \cite{1}. Based on our new model  formulas, the proposed algorithms  are very efficient for computing arbitrary GSV of an $(m,p,n)$-GMP $\{A,B\}$ for general scaled GSV problems: (i) For large $n$, we need design more effective algorithms to compute arbitrary GSV of the GMP for large scaled GSV problems. This is a new topic in the future. (ii) For $A^{H}A+B^{H}B$ being ill-posed (i.e., $\sigma_{\max}(A^{H}A+B^{H}B)$ is far away from $\sigma_{\min}(A^{H}A+B^{H}B)$), we need construct more effective algorithms. This is also a new topic in the future. (iii) The convergent analysis and theoretical computational complexity of the proposed algorithms are also discussed as a new topic in the future.
\end{remark}

\section{Numerical experiments}

In this section we give numerical experiments to illustrate the efficiency of Algorithms \ref{nm} and \ref{svd} for computing arbitrary GSV of a GMP. All the tests were carried out in {\tt MATLAB 2019a} running on a workstation with a Intel Xeon CPU E5-2687W at 3.10 GHz and 32 GB of RAM. In particular,  Algorithm \ref{svd}  was implemented by the  built-in functions {\tt svds} in {\tt MATLAB R2019a}.

In our numerical tests,  the tolerance for {\tt svds} is set to be $10^{-12}$ while the stopping criterion  for  Algorithm \ref{nm} is set to be
\[
\|P_{\Phi_{11}^k}D_{\Phi_{11}^k}\| \le 10^{-6}
\]
and we set $\eta=0.1$, $\rho=0.5$, and $\sigma=10^{-4}$. The largest number of outer iterations in Algorithm \ref{nm} is set to be $100$ and the largest number of iterations in the CG method is set to be $mn$ or $np$.
We will access efficiency of the proposed algorithms.

We first give some random examples to access the efficiency of the proposed algorithms.
\begin{example} \label{ex:1}
We first generate the exact generalized singular values $\alpha_1^*,\ldots,\alpha_n^*$ and $\beta_1^*,\ldots,\beta_n^*$ by using the  built-in function {\tt rand} in {\tt MATLAB R2019a} such that
\[
1\ge\alpha_1^*\ge\alpha_2^*\ge\cdots\ge \alpha_n^*\ge 0,\quad 0\le\beta_1^*\le\beta_2^*\le\cdots\le\beta_n^*\le 1,
\]
and
\[
\alpha_j^2+\beta_j^2=1,\quad 1\le j\le n.
\]
Then the $(m,p,n)$-GMP  $\{A,B\}$ is given by $A=U_*\Sigma_A^*W_*$ and $B=V_*\Sigma_B^*W_*$, where
\[
\Sigma_A^*=
\left(
\begin{array}{ccc}
\diag(\alpha_1^*,\ldots,\alpha_n^*) \\
O_{(m-n)\times n}
\end{array}
\right),\quad
\Sigma_B^*=
\left(
\begin{array}{ccc}
O_{(p-n)\times n} \\
\diag(\beta_1^*,\ldots,\beta_n^*)
\end{array}
\right),
\]
and the orthogonal matrices $U_*\in\mathbb{C}^{m\times m}$and $V_* \in\mathbb{C}^{p\times p}$ and the nonsingular matrix $W_*\in\Rnn$ are generated by the built-in functions {\tt orth} and {\tt randn} in {\tt MATLAB R2019a}. 
\end{example}

For demonstration purposes, we assume that $n\le m$ and $n\le p$. Table \ref{table41} lists the numerical results  for Example \ref{ex:1}, where `{\tt Err1.}', `{\tt Err2.}', `{\tt Err3.}',  `{\tt Err4.}',  and `{\tt CT.}' mean  the relative errors $\max_{1\le i\le [n/2]}|\alpha_i^k-\alpha_i^*|/|\alpha_i^*|$,  $\max_{[n/2]+1\le i\le n}|\beta_i^k-\beta_i^*|/|\beta_i^*|$, $\min_{1\le i\le [n/2]}|\alpha_i^k-\alpha_i^*|/|\alpha_i^*|$,  $\min_{[n/2]+1\le i\le n}|\beta_i^k-\beta_i^*|/|\beta_i^*|$, and the total computing time in seconds at the final iterates of the corresponding algorithms accordingly. 

\begin{table}[ht]\renewcommand{\arraystretch}{1.0} \addtolength{\tabcolsep}{2pt}
  \caption{Numerical results for Example \ref{ex:1}.}\label{table41}
  \begin{center} {\scriptsize
   \begin{tabular}[c]{|c|c|l|l|l|l|c|}
     \hline
Alg. & $(m,p,n)$   &  {\tt Err1.} &  {\tt Err2.} &  {\tt Err3.} &  {\tt Err4.} & {\tt CT.}\\  \hline

Alg. & (100, 80, 20)    & $1.02\cdot 10^{-15}$   & $2.15\cdot 10^{-15}$ & $1.79\cdot 10^{-16}   $   & $1.74\cdot 10^{-17}$ & $14.62$ \\
\ref{nm} & (200, 180, 40)     & $4.90\cdot 10^{-14}$   & $2.16\cdot 10^{-13}$ & $2.63\cdot 10^{-16}   $   & $2.70\cdot 10^{-16}$ & $19.80$ \\
& (500, 400, 80)     & $5.28\cdot 10^{-14}$   & $3.52\cdot 10^{-13}$ & $5.58\cdot 10^{-17}   $   & $2.58\cdot 10^{-16}$ & $20.35$ \\
& (800, 700, 100)   & $4.38\cdot 10^{-14}$   & $2.53\cdot 10^{-14}$  & $ 2.75\cdot 10^{-17}  $   & $3.85\cdot 10^{-17}$ & $21.18$ \\
& (1000, 800, 300)   & $5.06\cdot 10^{-14}$   & $3.72\cdot 10^{-14}$ & $ 2.24\cdot 10^{-16}     $   & $1.09\cdot 10^{-17}$ & $23.78$ \\
& (2000,1800, 500)  & $2.58\cdot 10^{-14}$   & $3.47\cdot 10^{-14}$ & $ 2.35\cdot 10^{-16}    $   & $3.06\cdot 10^{-16}$ & $30.05$ \\
& (3000,2500, 600)  & $3.20\cdot 10^{-13}$   & $3.11\cdot 10^{-14}$ & $ 1.83\cdot 10^{-15}  $   & $2.96\cdot 10^{-16}$ & $38.10$ \\
& (5000,4000, 600)   & $1.38\cdot 10^{-12}$   & $2.16\cdot 10^{-12}$ & $ 4.70\cdot 10^{-15}  $   & $2.63\cdot 10^{-14}$ & $42.07$ \\
& (10000,8000, 600)  & $1.84\cdot 10^{-13}$   & $5.26\cdot 10^{-14}$ & $ 3.28\cdot 10^{-15}   $   & $1.58\cdot 10^{-16}$ & $45.59$ \\
\cline{2-6}\hline

{\tt svds} & (100, 80, 20)    & $4.62\cdot 10^{-15}$   & $1.47\cdot 10^{-14}$ & $2.53\cdot 10^{-16}$   & $3.80\cdot 10^{-17}$ & $15.10$ \\
& (200, 180, 40)     & $6.28\cdot 10^{-14}$   & $7.90\cdot 10^{-13}$ & $3.37\cdot 10^{-16}$   & $4.07\cdot 10^{-16}$ & $20.03$ \\
& (500, 400, 80)     & $6.83\cdot 10^{-14}$   & $6.02\cdot 10^{-13}$ & $ 6.84\cdot 10^{-17}$   & $6.94\cdot 10^{-15}$ & $21.75$ \\
& (800, 700, 100)   & $5.40\cdot 10^{-14}$   & $3.11\cdot 10^{-14}$  & $ 3.01\cdot 10^{-17}$   & $5.57\cdot 10^{-17}$ & $22.94$ \\
& (1000, 800, 300)   & $7.55\cdot 10^{-14}$   & $5.09\cdot 10^{-14}$ & $2.36\cdot 10^{-16}$   & $8.46\cdot 10^{-16}$ & $24.80$ \\
& (2000,1800, 500)  & $4.05\cdot 10^{-14}$   & $5.38\cdot 10^{-14}$ & $ 2.96\cdot 10^{-16}$   & $5.88\cdot 10^{-16}$ & $24.52$ \\
& (3000,2500, 600)  & $4.67\cdot 10^{-13}$   & $5.93\cdot 10^{-14}$ & $ 2.74\cdot 10^{-15}$   & $4.27\cdot 10^{-16}$ & $25.09$ \\
& (5000,4000, 600)   & $4.52\cdot 10^{-12}$   & $3.85\cdot 10^{-12}$ & $ 2.07\cdot 10^{-14}$   & $1.44\cdot 10^{-13}$ & $25.95$ \\
& (10000,8000, 600)  & $3.08\cdot 10^{-13}$   & $1.05\cdot 10^{-13}$ & $6.93\cdot 10^{-15}$   & $5.96\cdot 10^{-16}$ & $26.03$ \\
\cline{2-6}\hline
\end{tabular} }
  \end{center}
\end{table}
\begin{figure}[ht]
\hfil\includegraphics[width=5.5in]{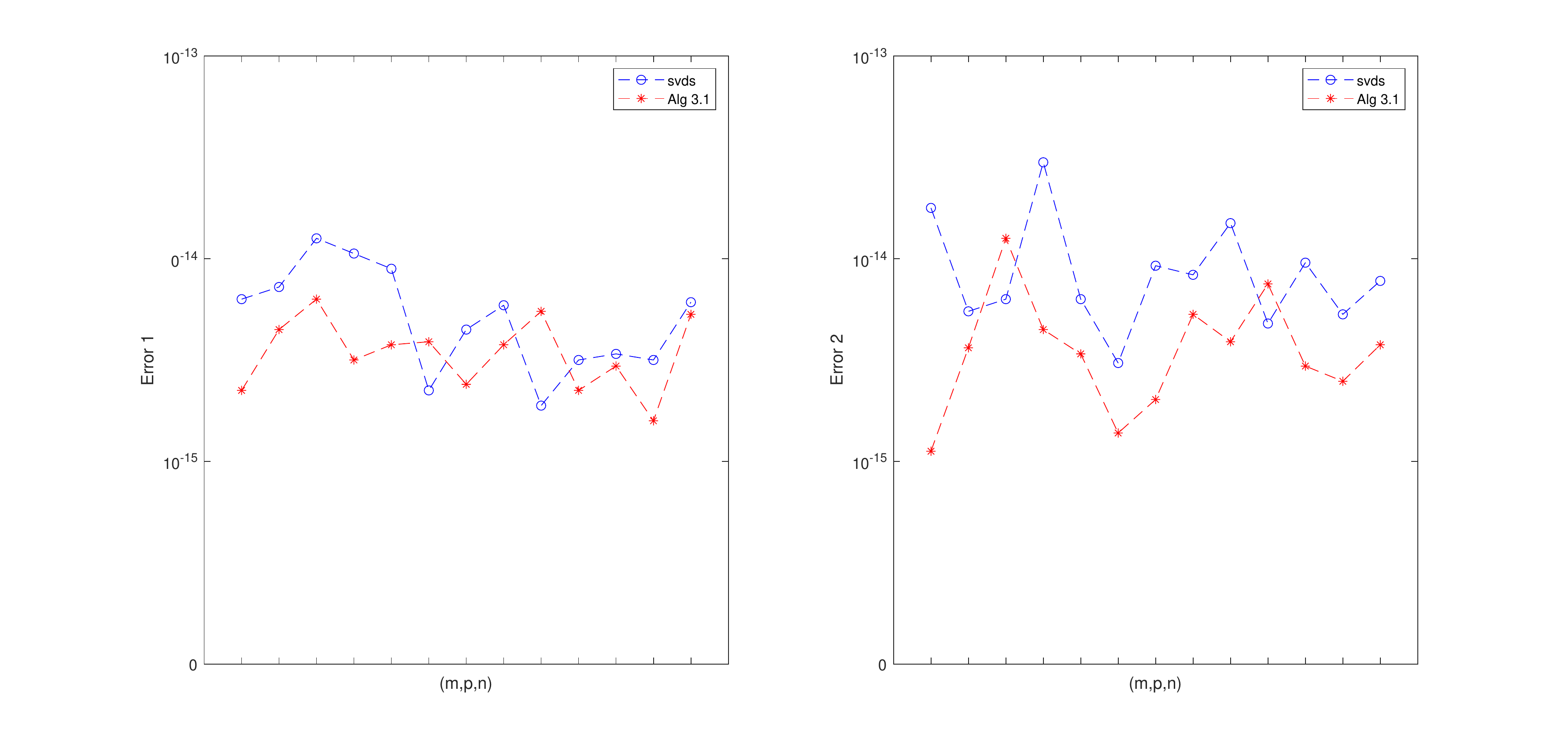}
\caption{Computed relative errors  for Example \ref{ex:1} with random GMPs} \label{fig41}
\end{figure}
\begin{figure}[ht]
\hfil\includegraphics[width=4in]{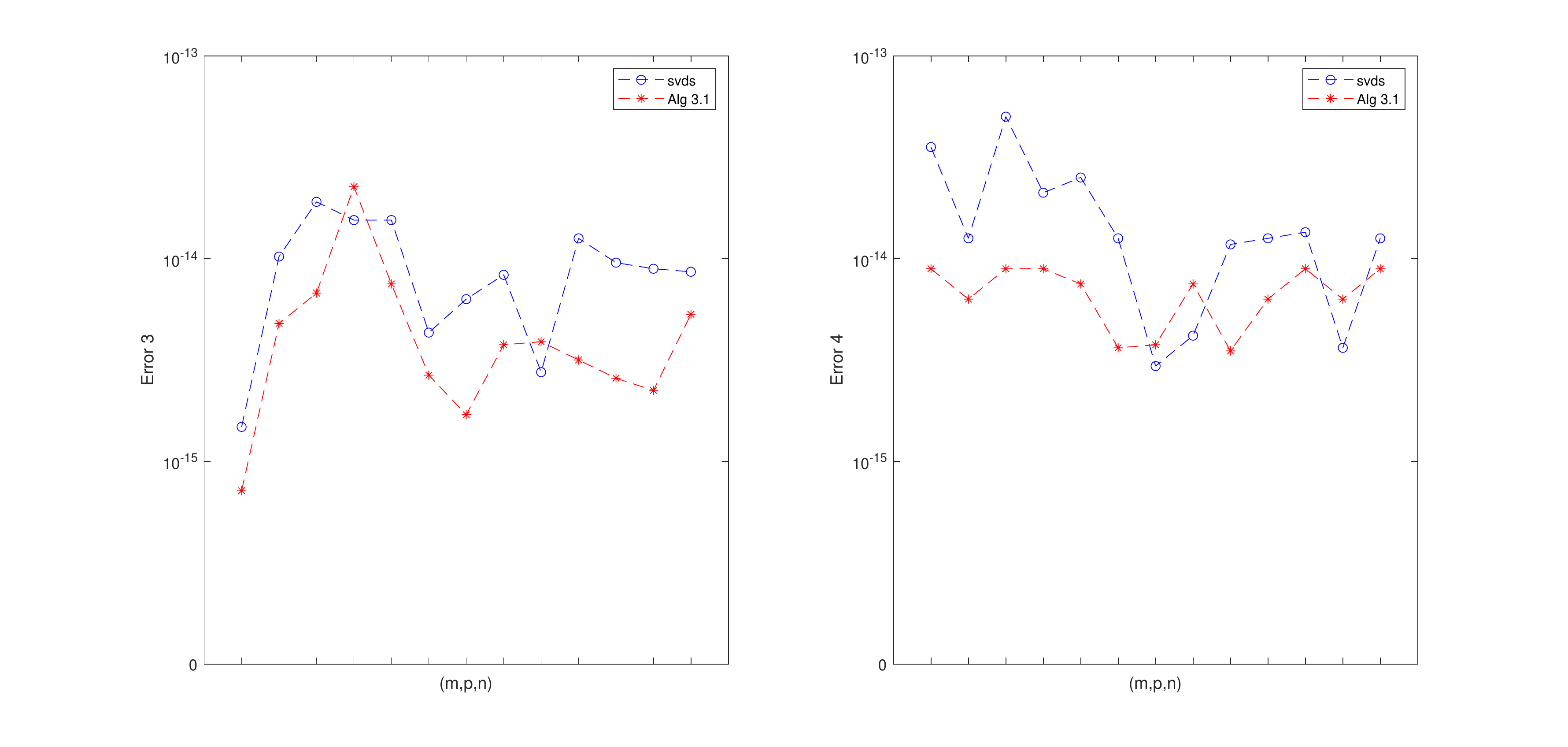}
\caption{Errors 3 and 4  for Example \ref{ex:1} with random GMPs} \label{fig42}
\end{figure}
\begin{figure}[ht]
\hfil\includegraphics[width=4in]{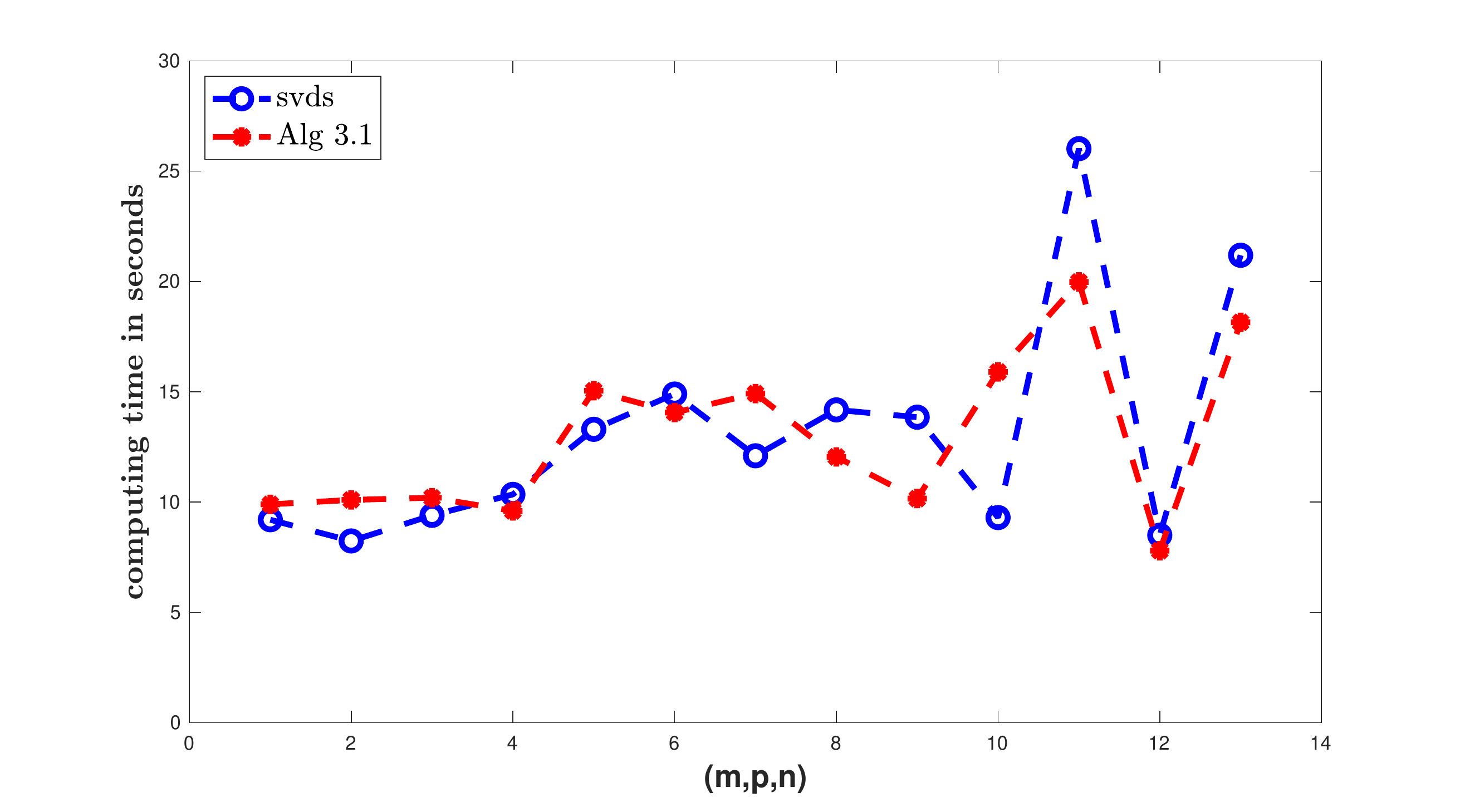}
\caption{Computing time  for Example \ref{ex:1} with random GMPs} \label{fig43}
\end{figure}

Also, in Figures \ref{fig41}--\ref{fig43}, we report the corresponding values of `{\tt Err1.}', `{\tt Err2.}', `{\tt Err3.}',  `{\tt Err4.}',  and `{\tt CT.}'  for Example \ref{ex:1} with other random $(m,p,n)$-GMPs.

We see from  Table \ref{table41} that the computational time by Algorithm \ref{nm} is less than $45.59$s and computational time by {\tt svds} is less than $26.03$s. From Table \ref{table41} and Figures \ref{fig41}--\ref{fig42},  we observe that all values of $\max_{1\le i\le  [n/2]}|\alpha_i^k-\alpha_i^*|/|\alpha_i^*|$, $\max_{ [n/2]+1\le i\le n}|\beta_i^k-\beta_i^*|/|\beta_i^*|$ and $\min_{1\le i\le  [n/2]}|\alpha_i^k-\alpha_i^*|/|\alpha_i^*|$, $\min_{ [n/2]+1\le i\le n}|\beta_i^k-\beta_i^*|/|\beta_i^*|$ computed by Algorithm \ref{nm} are sharper than those computed by {\tt  svds} for most cases. This implies the solution precision of Algorithm \ref{nm} is a little higher than the solver {\tt svds}. From Figure \ref{fig43} we see that the computational time of {\tt svds} is a little less than Algorithm \ref{nm} for some random GMPs.

We then give some practical numerical examples to access the efficiency of the proposed model and algorithms.
\begin{example} \label{ex:2}
We derive data of the corresponding gene mRNA expression level after doing the mice macrophage research experiment.
Several groups of datasets from different conditions with different chemical reagents, which are denoted by $\{E_{1},F_{1}\}$ and 
$\{E_{2},F_{2}\}$ as in Appendix.
\end{example}

As we know in \cite{4,5}, the GSVD induces a linear transformation of two
data sets from the two genes $\times$ arrays spaces to two reduced and
diagonalized genelets $\times$ arraylets spaces. The genelets are
shared by both data sets.
A single microarray probes the relative expression levels of $m$
genes in a single sample. A series of $n_{1}$ arrays probes the
genome-scale expression levels in $n_{1}$ different samples, i.e.,
under $n_{1}$ different experimental conditions. Let the matrix $A$,
of size $m$-genes $\times$ $n_{1}$-arrays, denote the full expression data, whose the $k$-th row is the
expression of the $k$-th gene across the different samples that correspond to the different arrays.

Let the matrix $B$, of size $p$-genes $\times$ $n_{2}$-arrays, denote the relative expression levels of
$p$ genes under $n_{2} = n_{1} = n <\max\{m, p\}$ experimental
conditions that correspond one to one to the $n_{1}$ conditions
underlying $A$.
The GSVD then induces simultaneous linear transformation of two
expression data sets $A$ and $B$ from two $m$-genes $\times$ $n$-arrays
and $p$-genes $\times$ $n$-arrays spaces to two reduced $n$-genelets
$\times$ $n$-arraylets spaces. For more details, please see the PNAS web site: www.pnas.org, and also http://genome-www.stanford.edu/GSVD/). Let $A$ and $B$ have the following GSVD:
\[
A=U\Sigma_{A}R^{-1}
\quad\mbox{and}\quad
B=V\Sigma_{B}R^{-1}.
\]
In these spaces the data is represented by the diagonal non-negative matrices $\Sigma_{A}$ and $\Sigma_{B}$ and denote $\langle k|\Sigma_{A}|l\rangle=\varepsilon_{1,l}\delta_{kl},\;  \langle k|\Sigma_{B}|l\rangle=\varepsilon_{2,l}\delta_{kl}$ for all $1\leq k,\;l\leq n$,
where $\delta_{kl}=1$ if $k=l$ and  $\delta_{kl}=0$ if $k\neq l$. By $\langle k|\Sigma_{A}$ we denote the $k$-th row of the matrix $\Sigma_{A}$, which lists the
expression of the $k$-th gene across the different samples that
correspond to the different arrays.  By $\Sigma_{A}|l\rangle$ we denote the $l$-th column of the matrix $\Sigma_{A}$, which lists the genome-scale
expression measured by the $l$-th array. The antisymmetric angular distance between the data sets,
\[
\theta_{l}=\arctan(\varepsilon_{1,l}/\varepsilon_{2,l})-\pi/4,
\]
indicates the relative significance of the mth genelet, i.e., its significance in the first data set relative to that in the second data set in
terms of the ratio of the expression information captured by this genelet in the first data set to that in the second data set. An angular distance of 0 indicates a genelet of equal significance in both data sets, with $\pm\frac{\pi}{4}\;(\pm0.785398163397448)$
indicates no significance in the second data set relative to the first, or in the first relative to the second, respectively. (For more details, please see the PNAS web site).

In this experiment expression levels of 10020 genes are probed in a single sample. A series of $n$ arrays probes the genome-scale expression levels in $n$ samples under different conditions. By $E_{1},F_{1},E_{2},F_{2}$ we denote gene mRNA expression data with original data set, with just PA stimulation and after knocking out SHP2 gene of mice without and with PA stimulation, respectively. Data sets $\{E_{1},F_{1}\}$ and $\{E_{2},F_{2}\}$ with $m=10020,p=10020$ and $n=128$ are drawn. The computed results for Example \ref{ex:2} are listed in Tables \ref{table421}--\ref{table423} and Figures \ref{fig51}--\ref{fig52}.

\textbf{Analysis steps:}

1. The GSVD induces simultaneous linear transformation of the two
expression data sets $\{E_{1}, F_{1}\}$ and $\{E_{2}, F_{2}\}$
from $10020$-genes $\times$ $n$-arrays spaces to the reduced $n$-genelets
$n$-arraylets spaces. By transforming into the GSVD, the data sets $\{E_{1}, F_{1}\}$ and $\{E_{2}, F_{2}\}$
are represented by the diagonal non-negative matrices $\{\Sigma_{E_{1}}, \Sigma_{F_{1}}\}$ and $\{\Sigma_{E_{2}}, \Sigma_{F_{2}}\}$. We denote
\begin{eqnarray}
&&\langle k|\Sigma_{E_{1}}|l\rangle=\varsigma_{E_{1},l}\delta_{kl},\;\;\langle k|\Sigma_{F_{1}}|l\rangle=\varsigma_{F_{1},l}\delta_{kl},\nonumber\\
&&\langle k|\Sigma_{E_{2}}|l\rangle=\varsigma_{E_{2},l}\delta_{kl},\;\;\langle k|\Sigma_{F_{2}}|l\rangle=\varsigma_{F_{2},l}\delta_{kl},\nonumber
\nonumber
\end{eqnarray}
for all $1\leq k,\;l\leq n$, respectively and the antisymmetric angular distance between the data sets
\begin{eqnarray}
\theta_{lE_{1},F_{1}}&=&\arctan(\alpha_{l}/\beta_{l})-\pi/4,\;\;
\theta_{lE_{2},F_{2}}=\arctan(\tilde{\alpha}_{l}/\tilde{\beta}_{l})-\pi/4.
\end{eqnarray}
We denote the GSVs of $\{E_{1},F_{1}\},\{E_{2},F_{2}\}$ by
\begin{eqnarray}
\sigma\{E_{1},F_{1}\}&=&\{\alpha_{i},\beta_{i}\},\;\;\sigma\{E_{2},F_{2}\}=\{\tilde{\alpha}_{i},\tilde{\beta}_{i}\},i=1,\ldots,n.\nonumber
\end{eqnarray}
We first compute the values of $\alpha_{1}$ and $\beta_{n}$.

2. If $\alpha_{1}\leq1,\beta_{n}\leq1$, we compute $\alpha_{t}$ for $t\ge 2$ and $\beta_{s}$ for $t\ge 2$ till to $\alpha_{t_{0}}<1,\beta_{s_{0}}<1.$ Then compute $\alpha_{j}$ for $t_{0}\le j\le [n/2]$ till to $\alpha_{j}=\frac{\sqrt{2}}{2}$.

3. If $\alpha_{j_{0}}=\frac{\sqrt{2}}{2}$ and $\alpha_{j_{0}+1}<\frac{\sqrt{2}}{2},$ then stop computing $\alpha_{j}$ for $j=j_{0}+2,\ldots,[n/2]$. If $\alpha_{j}>\frac{\sqrt{2}}{2},j=1,2,\ldots,[n/2],$ then compute $\beta_{k}$ for $k=n,n-1,\ldots,[n/2]+1$ till to $\beta_{k}=\frac{\sqrt{2}}{2}$.

On the other hand, if $\beta_{k_{0}}=\frac{\sqrt{2}}{2}$ and $\beta_{k_{0}-1}<\frac{\sqrt{2}}{2},$ then stop computing $\beta_{k}$ for $k=k_{0}-2,\ldots,[n/2]-1$. If $\beta_{k}>\frac{\sqrt{2}}{2}$ for $k=n,n-1,\ldots,[n/2]-1,$ then compute $\alpha_{j}$ for $j=1,2,\ldots,[n/2]$ till to $\alpha_{j}=\frac{\sqrt{2}}{2}$.

4. We need find $j$ for $\alpha_{j}=1,\frac{\sqrt{2}}{2}$ and $k$ for $\beta_{k}=1,\frac{\sqrt{2}}{2}$.

5. Meanwhile, by Theorem \ref{thm:2.5}, Lemma  \ref{lem:2.3}, and Lemma  \ref{lem:2.4}  we know that computing $\alpha_{j}$ and $\beta_{k}$ needs $Q_{j}$ and $\mathcal{G}_{j}$ for $j=1,2,\ldots,[n/2]$ and $P_{k}$ and $\mathcal{H}_{k}$ for $k=n,n-1,\ldots,[n/2]+1$. It refers to calculating only partial values of
$\alpha_{j}$ and $\beta_{k}$, which costs much less.

\vspace{1mm}
By the above analysis steps we compute the following $(\alpha_{i},\beta_{i})$ of $\{E_{1},F_{1}\}$ and $(\tilde\alpha_{i}, \tilde\beta_{i})$ of $\{E_{2},F_{2}\}$ for some $i$ (see  Tables \ref{table421}--\ref{table422}). We give computation steps for  $\{E_{1},F_{1}\}$ and $\{E_{2},F_{2}\}$ datasets for Example \ref{ex:2} as follows.

\textbf{Computation steps for $\{E_{1},F_{1}\}$ dataset:}

1. Firstly by using Algorithm \ref{nm} and {\tt svds} we compute $\alpha_{1}=0.810132020256465$ and $\beta_{n}=0.718697912381140$.

2. Then compute $\beta_{i},i=127,\ldots,121$ and $\beta_{121}=0.707160154056073\approx\frac{\sqrt{2}}{2}$ and continue to compute $\beta_{120}=0.705954228735626<\frac{\sqrt{2}}{2}.$

3. Since $\alpha_{i}\geq\alpha_{i+1}$ and $\beta_{i}\leq\beta_{i+1}$ for $i=1,\ldots,127$, we compute $\alpha_{1}=0.810132020256465$ and $\beta_{1}=0.586247481662121$ and
 \begin{eqnarray*}
&&0.718697912381140=\beta_{128}>\cdots>\beta_{122}=0.711863678067525>\beta_{121}\approx\frac{\sqrt{2}}{2}\nonumber\\
&&>\beta_{120}=0.705954228735626\geq\cdots\geq\beta_{1}=0.586247481662121.\nonumber
\end{eqnarray*}

4. We need find $j$ for $\alpha_{j}=1,\frac{\sqrt{2}}{2}$ and $k$ for $\beta_{k}=1,\frac{\sqrt{2}}{2}$. From the above analysis we know that $\beta_{121}=\frac{\sqrt{2}}{2}$ and for any $j,\alpha_{j}\neq1$ and $\beta_{j}\neq1$.

5. By $\alpha_{i}^{2}+\beta_{i}^{2}=1$ for $i=1,120,\ldots,128$ we can derive $\beta_{1}$ and $\alpha_{i}$ for $i=1,120,\ldots,128,$ and then compute $\theta_{lE_{1},F_{1}}=\arctan(\alpha_{l}/\beta_{l})-\pi/4=0$ for $l=121$ (see Table \ref{table423} and Figure \ref{fig51}).

6. Therefore, we need compute $\alpha_{1}$ and $\beta_{i}$ for $i=120,\ldots,128$, which needs $Q_{1},\mathcal{G}_{1},P_{k},\mathcal{H}_{k}$ for $k=120=n-8,\ldots,128=n$ and  the cost is much less by the proposed algorithms.

\begin{table}[ht]\renewcommand{\arraystretch}{1.0} \addtolength{\tabcolsep}{2pt}
  \caption{Computed $(\alpha_{i},\beta_{i})$  for $\{E_{1},F_{1}\}$ dataset for Example \ref{ex:2}.}\label{table421}
\begin{center} {\scriptsize
\begin{tabular}[c]{|c|c|r|l|}
     \hline
 Alg. & $i$  &  $\alpha_{i}$  & $\beta_{i}$ \\  \hline
  Alg. \ref{nm}   &     i=1  & 0.810132020256465&/ \\
  &i=128 & /&0.718697912381140               \\
  &     i=127 &/&0.717627620092847\\
  &       i=126&/&0.717063249543065\\
  &      i=125 &/&0.714466023589318\\
  &  i=124&/&0.713951890380495\\
  &   i=123&/&0.712111759572336\\
  &       i=122&/& 0.711863678067525\\
   &      i=121&/&0.707160154056073\\
     &      i=120&/&0.705954228735626\\
  &      &  &   \\
 {\tt svds} &   i=1  & 0.810111613718651  & /\\
&i=128  &/& 0.718688457942421\\
  &    i=127  &/&0.717625631885725\\
  &    i=126  &/&0.717060896558196\\
  &    i=125 &/&0.714449213982706\\
  &    i=124 &/&0.713940747495409\\
  &    i=123 &/&0.712104173213587\\
  &    i=122&/&0.711876696858187\\
   &     i=121&/& 0.707156111014042\\
     &   i=120&/&0.705959880773849\\\hline
  \end{tabular} }
  \end{center}
\end{table}
\begin{table}[ht]\renewcommand{\arraystretch}{1.0} \addtolength{\tabcolsep}{2pt}
  \caption{Computed  $(\tilde\alpha_{i}, \tilde\beta_{i})$ for $\{E_{2},F_{2}\}$ dataset for Example \ref{ex:2}.}\label{table422}
  \begin{center} {\scriptsize
  \begin{tabular}[c]{|c|c|r|l|l|}
     \hline
 Alg. & $i$  &  $\tilde\alpha_{i}$   & $\tilde\beta_{i}$ \\  \hline
  Alg. \ref{nm}  &    i=1  & 0.796382849891640  &/\\
   &i=128 &/&0.725245192340354\\
  &   i=127 &/&0.717466135452788\\
  &   i=126  &/&0.715890755680548\\
  &   i=125  &/&0.715628751189350\\
  &  i=124  &/&0.713702581283016\\
  &    i=123& /&0.712809336437327\\
  &   i=122&/&0.710632060519230\\
   &     i=121&/&0.709521730635737\\
     &  i=120&/&0.707850060237633\\
   &    i=119&/&0.707773758361755\\
  &   i=118&/&0.707066878837273 \\
  &    i=117&/&0.705546192429588 \\
     &      &  &   \\
  {\tt svds}  &   i=1  & 0.796382814200499 & /\\
 &i=128  &/ &0.725245187463159\\
  &  i=127  &/ &0.717466140240871\\
  &   i=126 &/ &0.715890779555181\\
  &   i=125 &/ &0.715628737273613\\
  &   i=124  &/ &0.713702591000907\\
  &   i=123& /& 0.712809355841902\\
  &   i=122& /& 0.710632056740559\\
   &   i=121&/&0.709521777784081\\
     &  i=120& /&0.707850078689362\\
   &     i=119&/ &0.707773720014017\\
  &   i=118&/ &0.707066856930138\\
  &    i=117& /&0.705546152277959 \\  \hline
  \end{tabular}}
  \end{center}
\end{table}

\begin{table}[ht]\renewcommand{\arraystretch}{1.0} \addtolength{\tabcolsep}{2pt}
  \caption{$\theta_{l}\{E_{1},F_{1}\}$ and $\theta_{l}\{E_{2},F_{2}\}$ for Example \ref{ex:2}.}\label{table423}
  \begin{center} {\scriptsize
  \begin{tabular}[c]{|c|c|r|l|}
     \hline
 $l$  &  $\theta_{l}\{E_{1},F_{1}\}$ &  $l$  & $\theta_{l}\{E_{2},F_{2}\}$\\ \hline
i=128&  -0.016529700047278&i=128&-0.025992294349443\\
 i=127& -0.014991647262518&i=127&-0.014759799059394\\
i=126&  -0.014181605218803 &i=126& -0.012500876385385\\
i=125& -0.010462462349320 &i=125& -0.012125698075447\\
i=124&  -0.009727915483202&i=124& -0.009371923296157\\
i=123& -0.007103397068787&i=123&-0.008097504074166\\
i=122& -0.006750101190428&i=122& -0.004998008671197\\
i=121 & 0&i=121 &-0.003421112938104\\
i=120 & 0.001628629810176&i=120& -0.001051708553944\\
 &$\vdots$&i=119&0.000943693585799\\
&$\downarrow$&  i=118&0\\
&$\vdots$&   i=117&0.002204577491173 \\
&$\vdots$&&$\vdots$\\
&$\downarrow\;increasing$&&$\downarrow\;increasing$\\
 i=1  & 0.158979112164203  &i=1&0.135892471613447\\
 \hline
  \end{tabular}}
  \end{center}
\end{table}
We then give computation steps for $\{E_{2},F_{2}\}$ dataset for Example \ref{ex:2} as follows.

\textbf{Computation steps for $\{E_{2},F_{2}\}$ dataset:}

1. Firstly, by using Algorithm \ref{nm} and {\tt svds} we compute $\alpha_{1}=0.796382849891640$ and $\beta_{n}=\beta_{128}=0.725245192340354$.

2. Then compute $\beta_{i}$ for $i=127,\ldots,118$ and $\beta_{118}=0.707066856930138\approx\frac{\sqrt{2}}{2}$ and continue to compute $\beta_{117}=0.705546152277959<\frac{\sqrt{2}}{2}.$

3. Since $\alpha_{i}\geq \alpha_{i+1}$ and $\beta_{i}\leq \beta_{i+1}$ for $i=1,\ldots,127$, we compute $\alpha_{1}=0.796382849891640$ and $\beta_{1}=0.604792821054011$ and
\begin{eqnarray*}
&&0.725245187463159=\beta_{128}>\cdots>\beta_{119}=0.707773720014017>\beta_{118}\approx\frac{\sqrt{2}}{2}\nonumber\\
&&>\beta_{117}=0.705546152277959\geq\cdots\geq\beta_{1}=0.604792821054011.\nonumber
\end{eqnarray*}

4. We need find $j$ for $\alpha_{j}=1,\frac{\sqrt{2}}{2}$ and $k$ for $\beta_{k}=1,\frac{\sqrt{2}}{2}$. From the above analysis we know that $\beta_{118}=\frac{\sqrt{2}}{2}$ and for any $j,\alpha_{j}\neq1$ and $\beta_{j}\neq1$.

5. By $\alpha_{i}^{2}+\beta_{i}^{2}=1$ for $i=1,117,\ldots,128$ we derive that $\beta_{1}$ and $\alpha_{i} $ for $i=1,117,\ldots,128,$ and then compute $\theta_{lE_{2},F_{2}}=\arctan(\tilde{\alpha}_{l}/\tilde{\beta}_{l})-\pi/4$ for $l=118$ (see Table \ref{table423} and Figure \ref{fig52}).

6. Therefore, we need compute $\alpha_{1}$ and $\beta_{i}$ for $i=117,\ldots,128$, which needs $Q_{1},\mathcal{G}_{1},P_{k},\mathcal{H}_{k}$ for $k=117=n-11,\ldots,128=n$ and also the cost is much less by the proposed algorithms.

\begin{figure}[ht]
\hfil\includegraphics[width=4in]{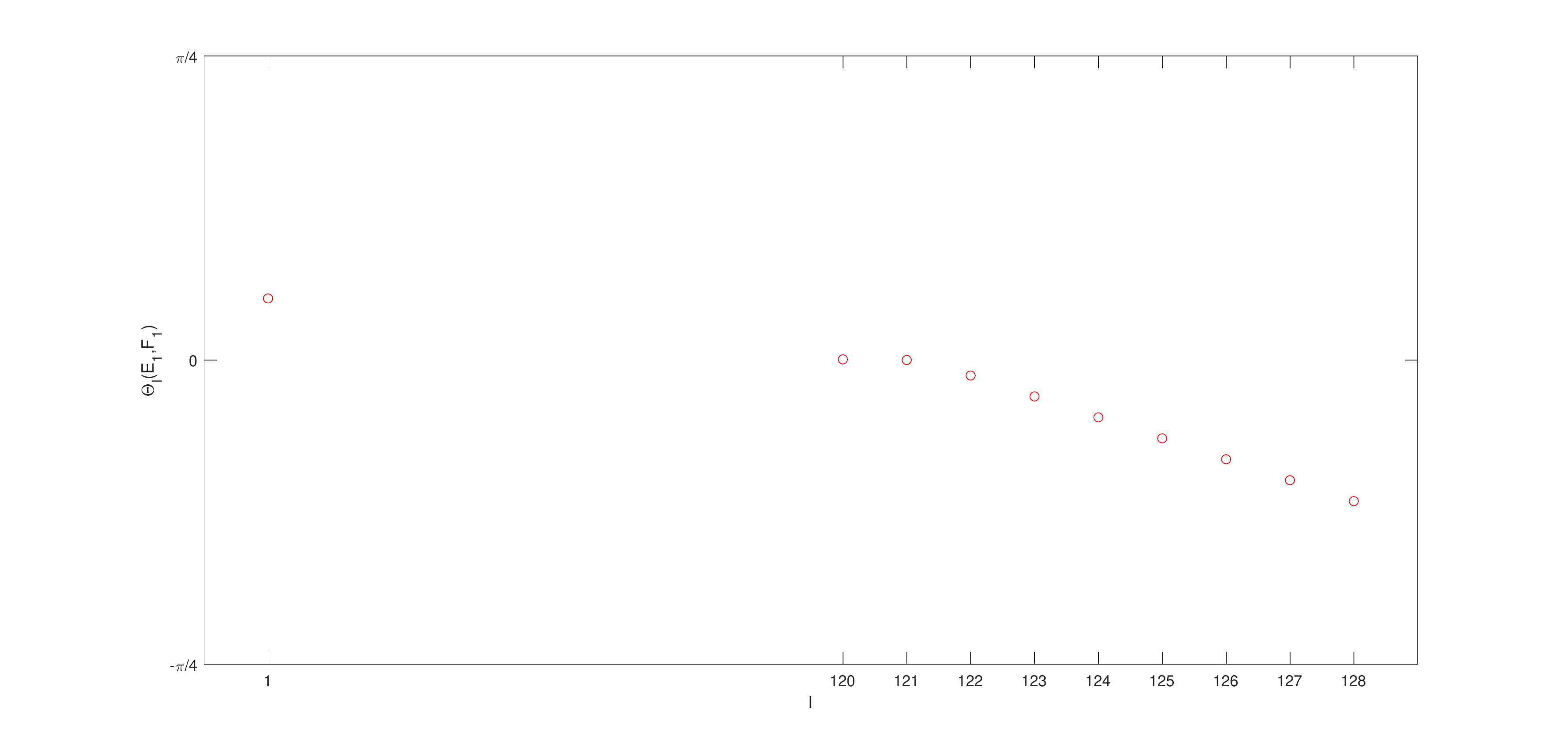}
\caption{$\theta_{l}\{E_{1},F_{1}\}$ for Example \ref{ex:2}.} \label{fig51}
\end{figure}

\begin{figure}[ht]
\hfil\includegraphics[width=4in]{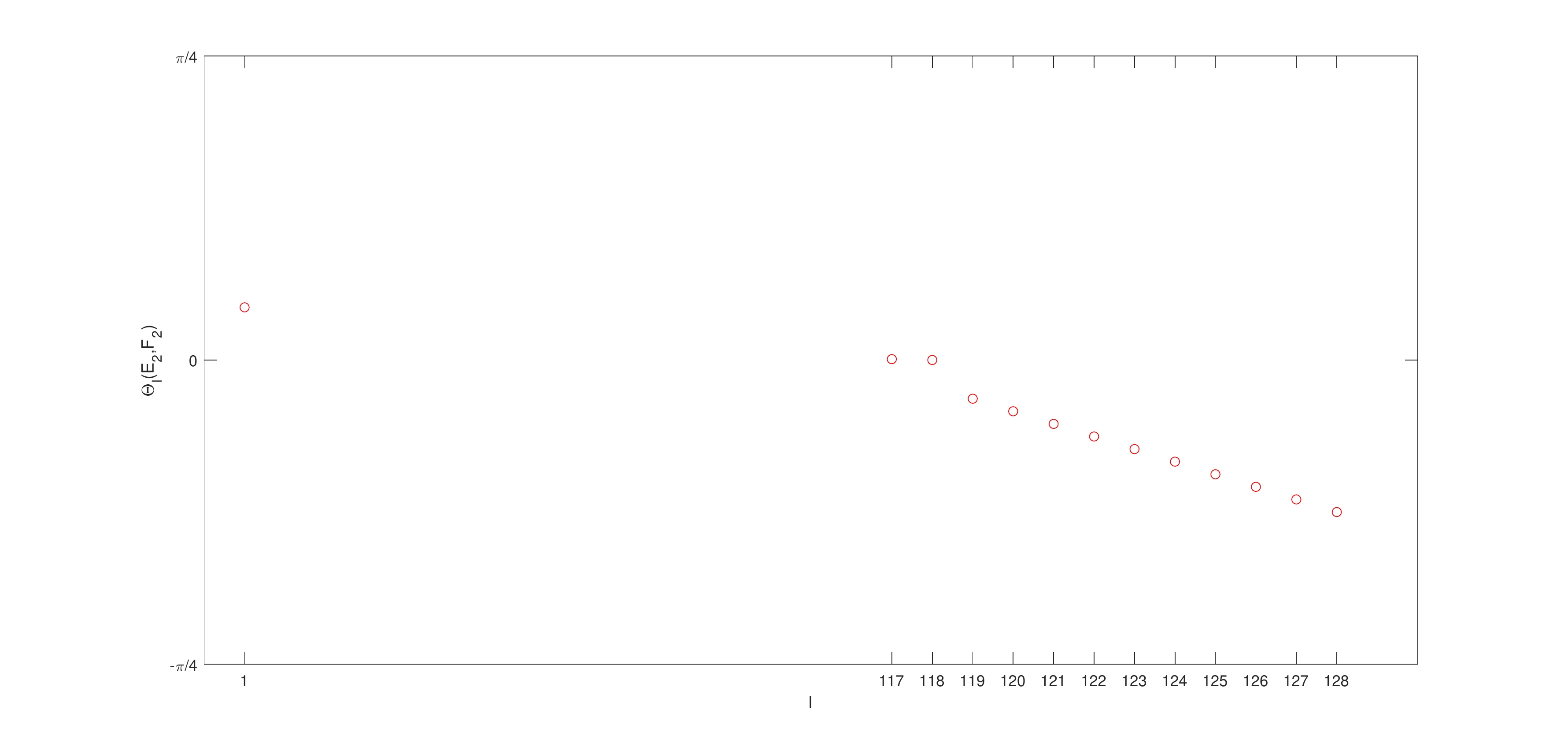}
\caption{$\theta_{l}\{E_{2},F_{2}\}$ for Example \ref{ex:2}.}\label{fig52}
\end{figure}

We conclude the following results.
\begin{itemize}
\item The computed values by Algorithm \ref{nm} and {\tt svds} are almost the same as Algorithm \ref{nm}.

\item For $l=121,\theta_{lE_{1},F_{1}}=0$ and for $l=118,\theta_{lE_{2},F_{2}}=0$. This shows $E_{1}$ and $F_{1}$ from original and PA stimulation data sets have genelets of equal significance in $l=121$-th genelet. $E_{2}$ and $F_{2}$ after knocking out SHP2 gene of mice without and with PA stimulation data sets have genelets of equal significance in $l=118$-th genelet. Hence we can conclude that $E_{1}$ and $F_{1}$ from  $E_{1}$-S1,$\ldots$, $E_{1}$-Sn and $F_{1}$-S1,$\ldots$, $F_{1}$-Sn data sets have genelets of equal significance in $121$-th genelet and $E_{2}$ and $F_{2}$ from $E_{2}$-S1,$\ldots$, $E_{2}$-Sn and $F_{2}$-S1,$\ldots$, $F_{2}$-Sn data sets have genelets of equal significance in $118$-th genelet. This is helpful and valuable for further experimental research when analyzing genelets of equal significance and no significance.
\end{itemize}

\section{Concluding remarks}
In this paper, from the aspect of small cost, high precision and practical applications, we give new formulations for computing arbitrary generalized singular value of Grassman matrix pair.
Then by using these new formula models, we design the corresponding algorithms involving Newton's method on Grassmann manifold and the {\tt MATLAB}-routine solver {\tt svds}, which only result in smaller cost and higher precision. We apply the proposed algorithms in some numerical tests. Numerical tests are accessed to illustrate the efficiency of the proposed methods for computing arbitrary generalized singular value of Grassman matrix pair. The advantage of the new formula models and the corresponding algorithms is to compute several generalized singular values, which cost much less than making the whole  generalized singular value decomposition

\section{Appendix}

Because the full data sets in Example 4.2 are big, we list a few lines from $E_{1},F_{1},E_{2},F_{2}$ data sets in Tables \ref{app:1}--\ref{app:4}. Full data sets will be attached as supplementary materials if needed.

\begin{table}[ht]\small
\caption{$E_{1}$ dataset} \label{app:1}
\begin{tabular}{ccccc}
\hline Gene-Symbol&$E_{1}-S1$& $E_{1}-S2$ &$E_{1}-S3$	&$\cdots$\\\hline
Dpm3
 &81.4726374728877 	&81.4726018003362 	&81.4727154775931 	&$\cdots$ \\
Gm15417
 & 90.5801106500692 &	90.5795202525307 &	90.5787761181889	&$\cdots$ \\
Pygo2
 &12.6975522059201 	&12.6988796188933 &	12.6993125974590 	&$\cdots$ \\
Shc1
 &91.3380167005621&	91.3378745358918&	91.3376629236392&$\cdots$ \\
Pbxip1
 &63.2360840051609	&63.2367024531678&	63.2364758975677&$\cdots$ \\
Pmvk
&9.75338665579280&	9.75526093535089&	9.75508758961278	&$\cdots$ \\
Kcnn3
 &27.8496050906937	&27.8505384828550 	&27.8511345102449 	&$\cdots$ \\
Adar
&54.6883232327317&	54.6891680197768 &	54.6886325094756 	&$\cdots$ \\
Ube2q1
&95.7524727418996	&95.7509942971894 &	95.7518196627301 	&$\cdots$ \\
She
&96.4902382384426 	&96.4892413334247 &	96.4893910841147 	&$\cdots$ \\
Gm19710
&15.7606332242848 &	15.7631050234608 &	15.7608553343750&$\cdots$\\
1700094D03Rik
&97.0607956377947&	97.0599299497393	&97.0594282912228	&$\cdots$ \\
4933434E20Rik
&95.7170575264070&	95.7173209697467&	95.7170535737553&$\cdots$ \\
$\cdots$&$\cdots$&$\cdots$&$\cdots$&$\cdots$\\\hline
\end{tabular}
\end{table}

\begin{table}[ht]\small
\caption{$F_{1}$ dataset} \label{app:2}
\begin{tabular}{ccccc}
\hline Gene-Symbol&$F_{1}-S1$& $F_{1}-S2$ &$F_{1}-S3$	&$\cdots$\\\hline
Dpm3 &81.4730879236440	&81.4730419716306 	&81.4733975198905 	&$\cdots$ \\
Gm15417& 90.5793236844437 &	90.5790411296398 &	90.5798411038575 	&$\cdots$ \\
Pygo2 &12.6986100545965	&12.6994178526520 &	12.6991540372817 	&$\cdots$ \\
Shc1 &91.3374209581048&	91.3383168331391&	91.3378505277527&$\cdots$ \\
Pbxip1 &63.2364298863660	&63.2368278125429&	63.2360364918807&$\cdots$ \\
Pmvk&9.75435696150952&	9.75486146024052&	9.75438342568573	&$\cdots$ \\
Kcnn3 &27.8504766511108	&27.8492428038682 	&27.8492926144878 	&$\cdots$ \\
Adar&54.6891748325179&	54.6887342933556 &	54.6887102842456 	&$\cdots$ \\
Ube2q1&95.7501784872693 	&95.7517688714067 &	95.7519690164486 	&$\cdots$ \\
She&96.4885801345460 	&96.4896650657055 &	96.4896621218625 	&$\cdots$ \\
Gm19710&15.7610321484248 &	15.7623919870229 &	15.7611206781601&$\cdots$\\
1700094D03Rik&97.0585851759412&	97.0606347263579	&97.0599898393001	&$\cdots$ \\
4933434E20Rik&95.7164506262917&	95.7158767256219&	95.7174835298262&$\cdots$ \\
$\cdots$&$\cdots$&$\cdots$&$\cdots$&$\cdots$\\\hline
\end{tabular}
\end{table}

\begin{table}[ht]\small
\caption{$E_{2}$ dataset} \label{app:3}
\begin{tabular}{ccccc}
\hline Gene-Symbol&$F_{1}-S1$& $F_{1}-S2$ &$F_{1}-S3$	&$\cdots$\\\hline
Dpm3 &81.4523098512597 	&81.4629821583036 	&81.4745727961033 	&$\cdots$ \\
Gm15417 & 90.5779819953193 &	90.5848577932900 &	90.5906502919661 	&$\cdots$ \\
Pygo2 &12.6974035147009	&12.6873751483440 &	12.7044923008073 	&$\cdots$ \\
Shc1 &91.3286343984840&	91.3251461714906&	91.3490834650614&$\cdots$ \\
Pbxip1 &63.2309324317487	&63.2334546316634&	63.2183204836416&$\cdots$ \\
Pmvk&9.75034064884895&	9.74695910760629&	9.76229363097876	&$\cdots$ \\
Kcnn3 &27.8621314570093	&27.8559191575373 	&27.8528545813737 	&$\cdots$ \\
Adar&54.7078492202499&	54.6885058089360 &	54.6830922478760 	&$\cdots$ \\
Ube2q1&95.7501045170701	&95.7625720815543 &	95.7713725989043 	&$\cdots$ \\
She&96.4900197059825 	&96.4971930401927 &96.4960354367868 	&$\cdots$ \\
Gm19710&15.7624053853472 &15.7593724639648&	15.7652870075486&$\cdots$\\
1700094D03Rik&97.0562110151379&	97.0573160527630	&97.0517544441638&$\cdots$ \\
4933434E20Rik&95.7235525356460&	95.7194450298556&	95.7137803601465&$\cdots$ \\
$\cdots$&$\cdots$&$\cdots$&$\cdots$&$\cdots$\\\hline
\end{tabular}
\end{table}

\begin{table}[ht]\small
\caption{$F_{2}$ dataset} \label{app:4}
\begin{tabular}{ccccc}
\hline Gene-Symbol&$F_{2}$-S1& $F_{2}$-S2 &$F_{2}$-S3	&$\cdots$\\\hline
Dpm3 &81.4817402134289 	&81.4749725312774 	&81.4812412409804 	&$\cdots$ \\
Gm15417 & 90.5840967987585 &90.5899209907963&	90.5746906609589 	&$\cdots$ \\
Pygo2 &12.6928679920493 	&12.7027778119699 &	12.7077668503773 	&$\cdots$ \\
Shc1 &91.3441570447946&	91.3397332006277&	91.3264759805061&$\cdots$ \\
Pbxip1 &63.2375980574507	&63.2321703530508&	63.2410869112732&$\cdots$ \\
Pmvk&9.75274288967974&	9.75412058445621&	9.74564742530634	&$\cdots$ \\
Kcnn3&27.8676517195286	&27.8522488517431 	&27.8484309535345 	&$\cdots$ \\
Adar&54.6778922856374&	54.6881209028444 &	54.6752501090303 	&$\cdots$ \\
Ube2q1&95.7382577903967 	&95.7358047141356 &	95.7488830100979 	&$\cdots$ \\
She&96.4874304231390 	&96.5078630839640 &	96.4794754795017 	&$\cdots$ \\
Gm19710&15.7754836985419 &15.7497320889512 &15.7634036450606&$\cdots$\\
1700094D03Rik&97.0546341778293&97.0583580985808	&97.0454377736231	&$\cdots$ \\
4933434E20Rik&95.7220563317885&	95.7103795510779&	95.7080583743294&$\cdots$ \\
$\cdots$&$\cdots$&$\cdots$&$\cdots$&$\cdots$\\\hline
\end{tabular}
\end{table}


\begin{thebibliography}{99}

\bibitem{AM08} P.-A. Absil, R. Mahony, and R. Sepulchre, Optimization Algorithms on Matrix Manifolds, Princeton University Press, Princeton, 2008.

\bibitem{4}
O. Alter, P. O. Brown, D. Botstein, Singular value decomposition for genome-wide expression
data processing and modelling, Proc. Natl. Acad. Sci. USA, 97 (2000), pp. 10101--10106.

\bibitem{5}
 O. Alter, P. O. Brown, D. Botstein, Generalized singular decomposition for comparative analysis
of genome-scale expression data sets of two different organisms, Proc. Natl. Acad. Sci. USA, 100 (2003), pp. 3351--3356.

\bibitem{6}
Z. Bai and J. W. Demmel, Computing the generalized singular value decomposition, SIAM J.
Sci. Comput., 14 (1993), pp. 1464--1486.


\bibitem{7}
J. L. Barlow, Error analysis and implementation aspects of deferred correction for equality constrained
least squares problems, SIAM J. Numer. Anal., 25 (1988), pp. 1340--1358.

\bibitem{8}
C. Davis, W. M. Kahan, and H. F. Weinberger, Norm-preserving dilations and their applications to optimal error bounds, SIAM J. Numer. Anal., 19 (1982), pp. 445--469.

\bibitem{9}
Z. Drma\v{c}, A tangent algorithm for computing the generalized singular value decomposition, SIAM J. Numer. Anal., 35 (1998), pp. 1804--1832.

\bibitem{13}
L. M. Ewerbring, F. T. Luk, Canonical correlations and generalized SVD: applications and new algorithms, J. Comput. Appl. Math., 27 (1989), pp. 37--52.

\bibitem{GK65} G. H. Golub and W. Kahan, Calculating the singular values and pseudo-inverse of a matrix, J. SIAM Ser. B Numer. Anal.,  2 (1965), pp. 205--224.

\bibitem{1} G. H. Golub, C. F. Van Loan, Matrix Computations, 4th ed., Johns Hopkins University Press, Baltimore, 2013.



\bibitem{2} R. A. Horn, C. R. Johnson, Matrix Analysis, 2nd ed., Cambridge University Press, Cambridge, 2013.

\bibitem{19}
L. K. Hua, Harmonic Analysis of Functions of Several Complex Variables in the Classical Domains, American Mathematical Society, Providence, 1963.

\bibitem{L50}  C. Lanczos, An iteration method for the solution of the eigenvalue problem of linear differential and integral operators, J. Res. Natl. Bur. Stand., 45 (1950), pp. 255--282.

\bibitem{10}
R. C. Li, Bounds on perturbations of generalized singular values and of associated subspaces,
SIAM J. Matrix Anal. Appl., 14 (1993), pp. 195--234.


\bibitem{27}
Q. K. Lu, The elliptic geometry of extended spaces, Acta Math. Sinica, 13 (1963), pp. 49--62.

\bibitem{M02} J. H. Manton, Optimization algorithms exploiting unitary constraints, IEEE Trans. Signal Process., 50 (2002), pp. 635--650.

\bibitem{26} Y. Matsushima, Differentiable Manifolds, Marcel Dekker, New York, 1972.

\bibitem{11}
C. C. Paige, Computing the generalized singular value decomposition, SIAM J. Sci. Stat. Comput., 7 (1986), pp. 1126--1146.

\bibitem{PS81}
C. C. Paige, M. A. Saunders, Towards a generalized singular value decomposition, SIAMJ.Numer. Anal., 18 (1981), pp. 398--405.
\bibitem{15}
G. W. Stewart,  Computing the CS-decomposition of a partitioned orthonormal matrix, Numer. Math., 40 (1982), pp. 297--306.

 \bibitem{23}
J.-G. Sun, Perturbation analysis for the generalized singular value problem, SIAM J. Numer. Anal., 20 (1983), pp. 611--625.


\bibitem{3}  C. F. Van Loan, Generalizing the singular value decomposition, SIAM J. Numer. Anal., 13 (1976), pp. 76--83.


\bibitem{16}
C. F. Van Loan, Computing the CS and the generalized singular value decompositions, Numer. Math., 46 (1985), pp. 479--491.


\bibitem{17}
W.-W. Xu, H. K. Pang, W. Li, X. P. Huang, W. J. Sun, On the explicit expression of chordal metric between generalized singular values of Grassmann matrix pairs with applications, SIAM J. Matrix Anal. Appl., 39 (2018), pp. 1547--1563.


\bibitem{18}
W.-W. Xu, W. Li, L. Zhu, X. P. Huang, The analytic solutions of a class of constrained matrix minimization and maximization problems with applications, SIAM J. Optim., 29 (2019), pp. 1657--1686.


\bibitem{14}
H. Zha, A numerical algorithm for computing restricted singular value decomposition of matrix triplets,
Linear Algebra Appl., 168 (1992), pp. 1--25.


\bibitem{ZB15} Z. Zhao,  Z. J. Bai, and X. Q. Jin, A Riemannian Newton algorithm for nonlinear eigenvalue problems, SIAM J. Matrix Anal. Appl., 36 (2015), pp. 752--774.


\end{thebibliography}
\end{document}